\documentclass[10pt]{amsart}
\pdfoutput=1
\usepackage{amsmath}
\usepackage{amsfonts}
\usepackage{amssymb}
\usepackage{amsthm, enumerate, enumitem}
\usepackage{hyperref}
\usepackage{comment}
\usepackage{float}
\usepackage{youngtab}
\usepackage{ytableau}
\usepackage{rotating}
\usepackage{subcaption}
\usepackage{xcolor}
\usepackage{bbding}
\usepackage{tikz}
\usepackage{tikz-cd}
\usepackage[paper=a4paper, margin=3.5cm]{geometry}

\def\ZZ{{\NZQ Z}}
\def\RR{{\NZQ R}}

\def\PP{{\NZQ P}}

\def\frk{\frak}               

\def\Phi{{\frk n}}
\def\Phi{{\frk N}}
\def\fF{{\mathcal F}}

\def\opn#1#2{\def#1{\operatorname{#2}}} 
\opn\chara{char} \opn\length{\ell} \opn\pd{pd} \opn\rk{rk}
\opn\projdim{proj\,dim} \opn\injdim{inj\,dim} \opn\rank{rank}
\opn\spn{span}\opn\Seg{Seg}
\opn\depth{depth} \opn\grade{grade} \opn\height{height}
\opn\embdim{emb\,dim} \opn\codim{codim}

\opn\Tr{Tr} \opn\bigrank{big\,rank}
\opn\superheight{superheight}\opn\lcm{lcm}
\opn\trdeg{tr\,deg}
\opn\reg{reg} \opn\lreg{lreg} \opn\ini{in} \opn\lpd{lpd}
\opn\size{size}\opn\bigsize{bigsize}
\opn\cosize{cosize}\opn\bigcosize{bigcosize}
\opn\sdepth{sdepth}\opn\sreg{sreg}
\opn\link{link}\opn\fdepth{fdepth} \opn\trdeg{trdeg} \opn\mod{mod}
\opn\spann{span}
\opn\div{div} \opn\Div{Div} \opn\cl{cl} \opn\Cl{Cl}
\opn\Spec{Spec} \opn\Supp{Supp} \opn\supp{supp} \opn\Sing{Sing}
\opn\Ass{Ass} \opn\Min{Min}\opn\Mon{Mon} \opn\dstab{dstab} \opn\astab{astab}
\opn\Syz{Syz}
\opn\Ann{Ann} \opn\Rad{Rad} \opn\Soc{Soc} \opn\Aut{Aut}
\opn\Im{Im} \opn\Ker{Ker} \opn\Coker{Coker} \opn\Am{Am}
\opn\Hom{Hom} \opn\Tor{Tor} \opn\Ext{Ext} \opn\End{End}
\opn\Aut{Aut} \opn\id{id}

\opn\nat{nat}
\opn\pff{pf}
\opn\Pf{Pf} \opn\GL{GL} \opn\SL{SL} \opn\mod{mod} \opn\ord{ord}
\opn\Gin{Gin} \opn\Hilb{Hilb}\opn\sort{sort}
\opn\S{S} \opn\dim{dim} \opn\supp{supp}\opn\trdeg{trdeg}\opn\sort{sort}
\opn\aff{aff} \opn\con{conv} \opn\relint{relint} \opn\st{st}
\opn\lk{lk} \opn\cn{cn} \opn\core{core} \opn\vol{vol}
\opn\link{link} \opn\star{star}\opn\lex{lex}
\opn\conv{conv} \opn\Ehr{Ehr}\opn\Pic{Pic}
\opn\Conv{Conv}
\opn\gr{gr}

\def\pot#1#2{#1[\kern-0.28ex[#2]\kern-0.28ex]}

\opn\dirlim{\underrightarrow{\lim}}
\opn\inivlim{\underleftarrow{\lim}}
\def\Implies{\ifmmode\Longrightarrow \else
        \unskip${}\Longrightarrow{}$\ignorespaces\fi}
\def\implies{\ifmmode\Rightarrow \else
        \unskip${}\Rightarrow{}$\ignorespaces\fi}
\def\iff{\ifmmode\Longleftrightarrow \else
        \unskip${}\Longleftrightarrow{}$\ignorespaces\fi}

\let\:=\colon

\newtheorem{Theorem}{Theorem}[section]
 \newtheorem{Lemma}[Theorem]{Lemma}
 \newtheorem{Corollary}[Theorem]{Corollary}
 
 \newtheorem{Remark}[Theorem]{Remark}

 \newtheorem{Example}[Theorem]{Example}
 
 \newtheorem{Definition}[Theorem]{Definition}

\def\RR{\mathbb{R}}
\def\ZZ{\mathbb{Z}}
\def\PP{\mathbb{P}}

\def\fC{\mathcal{C}}

\begin{document}
 \title {Phylogenetic degrees for Jukes-Cantor model}
 
 \author{Rodica Andreea Dinu}
\address{%
	University of Konstanz, Fachbereich Mathematik und Statistik, Fach D 197 D-78457 Konstanz, Germany, and Simion Stoilow Institute of Mathematics of the Romanian Academy, Calea Grivitei 21, 010702, Bucharest, Romania}
	\email{rodica.dinu@uni-konstanz.de}

\author{Martin Vodi\v{c}ka}
\address{Šafárik University, Faculty of Science, Jesenná 5, 04154 Košice, Slovakia}
	\email{martin.vodicka@upjs.sk}

\maketitle

 \begin{abstract}
Jukes-Cantor model is one of the most meaningful statistical models from a biological perspective. We are interested in computing the algebraic degrees for phylogenetic varieties, which we call \textit{phylogenetic degrees}, associated to the Jukes-Cantor model and any tree. As these varieties are toric, their geometry is hidden in the associated polytopes. For this reason, we provide two different combinatorial approaches to compute the volume for these polytopes. 
 \end{abstract}

 \maketitle

\section{Introduction}

One of the main goals of phylogenetics is to reconstruct historical relationships between species by examining current traits and placing their common ancestors on a tree, usually called \textit{phylogenetic tree}. The living species are represented at the leaves of the tree and the inner nodes represent ancestral sequences.

The study of phylogenetics, and in particular of evolutionary biology, has given rise to some fascinating new mathematical problems. There are many connections to different areas of mathematics, e.g., algebraic geometry \cite{bw, erss}, mathematical physics through conformal blocks \cite{manon}, and combinatorics \cite{matJCTA, RM-claw}. In the articles \cite{AR1, AR2}, Allman and Rhodes introduced algebraic varieties coming from phylogenetics which are called \textit{phylogenetic varieties}. In broad terms, phylogenetic varieties associated to a Markov model and a tree are the smallest algebraic varieties containing the set of joint distributions of nucleotides at the leaves of the tree. Its biological significance stems from the fact that the theoretical joint distribution of the nucleotides of the current species will always be represented by a point in this phylogenetic variety, regardless of the statistical parameters. The construction of these varieties will be presented in this article. In addition, we will assume that an abelian group $G$ (which will be, in fact, $\ZZ_2$ in our case) acts transitively and freely on the set of states and this fact gives rise to the so-called \textit{group-based models} and their associated phylogenetic group-based varieties. For more details regarding group-based models and their geometrical properties, the reader may consult \cite{mateusz, mateusztor, martinko, RM}. Readers interested in more algebraic varieties arising from statistical models may consult the monograph on algebraic statistics \cite{sethbook}.

We are interested in studying the algebraic degrees for group-based phylogenetic varieties. These degrees will be called \textit{phylogenetic degrees}. So far, in the literature, there are known concrete formulas for these degrees when the group is one of $\{\ZZ_2, \ZZ_2\times \ZZ_2, \ZZ_3\}$ and the tree is an $n$-claw tree (i.e. a tree which has exactly one inner vertex of degree $n$, which is also called \textit{star}), see \cite{RM-claw}. We refer also to the paper \cite{examples} for some computational results for small trees.

In this article, we investigate the phylogenetic degrees of the projective variety associated to the Jukes-Cantor model (i.e. when $G=\ZZ_2$) and any tree. Together with the 3-Kimura parameter model, this is one of the most meaningful statistical models in biology. The Jukes-Cantor model calculates the likelihood of substitution between two states (nucleotides were used in the original model, but codons or amino acids can also be used as a substitute). Interesting results known in the literature about the phylogenetic varieties associated to the Jukes-Cantor model can be found it \cite{ss, erss, bw,manon, RM-claw, matem}. The phylogenetic variety associated to the Jukes-Cantor model and any tree, which will be denoted by $X_T$, is a projective toric variety, \cite{ss}. Hence, the problem of computing the phylogenetic degrees can be translated into computing the volume for their associated polytope, $P_T$.  For a monograph on triangulations and computing volumes of polytopes, we refer to \cite{triangbook}. 


We present here the structure of the paper. In Section~\ref{prel} we introduce phylogenetic group-based varieties and their associated polytopes, together with notations that will be used further. In Section~\ref{Qr} we work with forests $F$ with edge sets $E(F)$. We introduce a family of polytopes $Q_F$ obtained by cutting the unit cube $[0,1]^{E(F)}$ by suitable hyperplanes corresponding to the inner vertices of the forest. We provide two recursive formulas for computing the volume of polytopes $Q_F$. One is given in Lemma~\ref{P-formula} and we illustrate on concrete examples such as the star and double-star graphs, and the other one is given in Corollary~\ref{P-splitting-formula}. 
Based on the construction of polytopes $Q_F$ and Lemma~\ref{beautiful} which is of high importance, we define a function $r(F)$ given by taking the sum of all the volumes of polytopes $Q_F$ for which we contract all the inner edges (i.e. the edges connecting two inner vertices) and we transform them in leaves. In Theorem~\ref{V-formula}, we express the volume of $P_T$ using the function $r(F)$. This gives us a three-step process of computing the volume of $P_T$. First, we compute volumes of polytopes $Q_F$, then the numbers $r(F)$ and finish with the volume of $P_T$.  In Corollary~\ref{Vstar} and Example~\ref{example-V} we provide suitable examples for his approach. In Section~\ref{anotherformula}, we provide another recursive formula, in Theorem~\ref{v-splitting}, for computing the volumes of $P_T$, this time without the use of the polytopes $Q_F$ and the function $r(F)$. Sometimes it is fast to compute the volume using this formula, and we illustrate it for some specific graphs, see Corollary~\ref{easy-Vformula} and Corollary~\ref{easy_formula_2}.

\section{Preliminaries}\label{prel}
We introduce the notation and preliminaries that will be used in this article. For more details, we refer to \cite{ss, erss, h-rep, RM}.

\subsection{Phylogenetic trees}\label{phylotrees}
We introduce the algebraic variety associated to a phylogenetic tree and a group, and its defining polytope. The construction of this variety with all the details may be found in the seminal paper \cite{erss}. The reader may also consult \cite{bw}.

 Let $T$ be a rooted tree with a set of vertices $V=V(T)$ and a set of edges $E=E(T)$. We direct the edges of $T$ away from the root. A vertex $v$ is called a {\it leaf} if its valency is 1, otherwise it is called a {\it node}. We denote the set of leaves in $T$ by $L$ and the set of nodes in $T$ by $N$. Let $S=\{\alpha_0, \alpha_1, \dots\}$ be a finite set, that is usually called the set of states, and let $W$ be a finite-dimensional vector space spanned freely by $S$. Let $\widehat{W}$ be a subspace of $W\otimes W$. Any element of $\widehat{W}$ can be represented as a matrix in the basis corresponding to $S$ and this matrix is called a {\it transition matrix}. We associate to any vertex $v\in V$ a complex vector space $W_v \simeq W$, and we associate to any edge $e\in E$ a subspace $\widehat{W}_e \subset W_{v_1(e)}\otimes W_{v_2(e)}$, $\widehat{W}_e \simeq \widehat{W}$, where $e$ is the directed edge from the vertex $v_1(e)$ to $v_2(e)$.
We call the tuple $(T, W, \widehat{W})$ a \textit{phylogenetic tree}.

 \subsection{A variety associated to a group-based model.}
We define the following spaces:
\[
W_V=\bigotimes_{v\in V} W_v,\;\;  W_L=\bigotimes_{l\in L} W_l,\;\;  \widehat{W}_E= \bigotimes_{e\in E} \widehat{W}_e.
\]
The space $W_V$ is called the space of all possible states of the tree, $W_L$ is called the space of states of leaves, and $\widehat{W}_E$ is the parameter space.

\begin{Definition} [{\cite[Construction 1.5]{bw}}]
We consider a linear map $\widehat{\psi}: \widehat{W}_E \rightarrow W_V$ whose dual is defined as
\[
\widehat{\psi}^{*}(\bigotimes_{v\in V} \alpha_v^{*})= \bigotimes_{e\in E}(\alpha_{v_1(e)}\otimes \alpha_{v_2(e)})^{*}_{| \widehat{W}_e},
\]
where $\alpha_v$ is an element of $S$.
\end{Definition}
The map $\widehat{\psi}$ associates to a given choice of matrices the probability distribution on the set of all possible states of vertices of the tree. Its associated multi-linear map $\widetilde{\psi}: \prod_{e\in E}\widehat{W}_e \rightarrow W_V$ induces a rational map of projective varieties
\[
\psi: \prod_{e\in E}(\PP(\widehat{W}_e))\dashrightarrow \PP(W_V).
\]
We consider the map $\pi: W_V \rightarrow W_L$ which is defined as
\[
\pi=(\otimes_{v\in L}\id_{W_v})\otimes (\otimes_{v\in N} \sigma_{W_v}),
\]
with $\sigma=\sum \alpha_i^{*} \in W^{*}$ which sums up all the coordinates. The map $\pi$ sums up the probabilities of the states of vertices which are different only in nodes.

\begin{Definition}\label{variety}
The image of the composition $\pi \circ \psi$ is an affine subvariety in $W_L$ and it is called the affine variety of a phylogenetic tree $(T, W, \widehat{W})$. By $X_T$ we denote its underlying projective variety in $\PP(W_L)$.
\end{Definition}

We will assume that an abelian group $G$ acts transitively and freely on the set $S$. We define $\widehat{W}$ to be the set of all fixed points of the $G$ action on $W\otimes W$. In this case, the space $\mathcal{M}$ of transition matrices is given as a subspace invariant under the action of the group $G$. This choice of $\widehat{W}$ leads to the so-called {\it group-based models}. For more details about these models, the reader may consult \cite{ss, erss} and \cite{mateusztor}. The variety from Definition~\ref{variety} is called the \textit{group-based phylogenetic variety} and it depends on the choice of the abelian group $G$.

\begin{Example}
If $G=\ZZ_2$, then we obtain the (two-state) Jukes-Cantor model, known in the literature also as the Cavender-Farris-Neyman model. The elements of $\widehat{W}$ represented as matrices are of the following form: 
\[
\begin{pmatrix}
a & b \\
b & a 
\end{pmatrix}.
\]
\end{Example}

In this article, we will consider only the Jukes-Cantor model and its associated group-based phylogenetic variety will be denoted by $X_T$.

\subsection{The defining polytope.}
By more general results due to Hendy and Penny \cite{HP} and, later, Erdős, Steel, and Székely \cite{sz}, the group-based phylogenetic varieties are projective toric varieties. In particular, our phylogenetic variety associated to the Jukes-Cantor model is a projective toric variety. Also, it is known that the geometric properties of toric varieties are hidden in those of their associated polytopes. For this reason, we will work with the defining polytope of the group-based phylogenetic variety associated to the Jukes-Cantor model, and we denote it by $P_T$.

We introduce some notation that will be used in the next sections when working with trees and polytopes $P_{T}\subset\mathbb R^{|E|}$. For a tree $T$ we denote the set of its edges by $E(T)$ and the set of inner vertices by $I(T)$. For any inner vertex $v\in I(T)$ we will denote by $\mathcal{N}_T (v)$ the set of edges that are adjacent to $v$, and, by $IE(T)$ the set of inner edges, i.e. the edges connecting two inner vertices. We will denote by $S_n$ the star graph, i.e. the tree with $n$ edges and exactly one inner vertex (also called in the literature as the \textit{$n$-claw tree}). We will denote by $S_{m,n}$ the double-star graph, i.e the tree with exactly two inner vertices with degrees $m+1$ and $n+1$.

For a set $Y$, we denote by $\mathcal{P}(Y)$ the set of all subsets of $Y$ and by $\mathcal{P}_{odd}(Y)$ the set of all subsets of $Y$ having odd cardinality. We label the coordinates of a point $x\in\mathbb R^{|E|}$ by $x_e$, where $e\in |E|$ corresponds to an edge of the tree.

The vertex and facet descriptions of the polytopes $P_{T}$ are known and, even more general results, may be found in \cite{bw, mateusz, ss}.
Here we recall these descriptions:


\begin{Theorem} 
The vertices of the polytope $P_{T}$ associated to the group $\ZZ_2$ and the tree $T$ are exactly the points from the following set:

$$\{x\in\{0,1\}^{|E(T)|}: \forall v\in I(T): 2\mid\sum_{e\in \mathcal N_T(v)} x_e\}$$.

In other words, in a vertex of $P_T$ the sum of edge coordinates around each inner vertex must be even.

Let $L_{T}$ be the lattice generated by the vertices of $P_{T}$. Then
 $$L_{T}=\{x\in \mathbb Z^{|E(T)|}: \forall v\in I(T): 2\mid\sum_{e\in \mathcal N_T(v)} x_e \}.$$
Moreover, the lattice $L_T$ is the sublattice of $Z^{|E|}$ of index $2^{|I(T)|}$.
\end{Theorem}

\begin{Example}
Let us consider the double-star $S_{2,2}$. Then
 \begin{align*}
P_{T}=\Conv(\{&(0,0,0,0,0),(1,1,0,0,0),(0,0,0,1,1),(1,1,0,1,1),\\
&(1,0,1,1,0),(1,0,1,0,1),(0,1,1,1,0),(0,1,1,0,1)\}).
\end{align*}
The first three coordinates correspond to the edges adjacent to the first inner vertex and the last three coordinates correspond to the edges adjacent to the second inner vertex. Thus, the third coordinate corresponds to the inner edge.

The lattice generated by the vertices of the polytope $P_{T}$ is
\begin{align*}
L_{T}&=\{x\in\mathbb Z^5: 2\mid x_1+x_2+x_3,\ x_3+x_4+x_5\}.
\end{align*}
\end{Example}

 It is known that the facet representation of $P_T$ is the following one:

\begin{Theorem}[{\cite{bw}}]\label{facetdescription_Z2}
The facet description of the polytope $P_{T}$ is:
\begin{itemize}
\item $0\le x_e\leq 1$ for all $e\in E(T)$,
\item For all $v\in I(T)$ and for all $A\in\mathcal{P}_{odd}(\mathcal{N}_T(v))$:
\[
\sum_{e\in\mathcal{N}_T(v)\setminus A} x_e - \sum_{e\in A} x_e\ge 1-|A|.\]

\end{itemize}
\end{Theorem}

We set also the following notations. For $v\in I$, $A\in\mathcal{P}(\mathcal{N}_T(v))$ we will denote by 
$$S_{v,A}(x):=\sum_{e\in\mathcal{N}_T(v)\setminus A} x_e - \sum_{e\in A} x_e,$$
$$H_{v,A}:=\{x\in \RR^{|E|}: S_{v,A}=1-|A|\},$$
$$H^+_{v,A}:=\{x\in \RR^{|E|}: S_{v,A}\ge 1-|A|\},$$
$$H^-_{v,A}:=\{x\in \RR^{|E|}: S_{v,A}\le 1-|A|\}.$$

\vspace{.1in}
\section{Phylogenetic degrees}\label{Qr}
\vspace{.2in}

In this section, we compute the phylogenetic degrees of varieties $X_T$. By \cite[Section 5.3]{fulton}, computing the phylogenetic degrees for toric varieties relies on computing lattice volumes for the associated polytopes $P_T$. 
 For a polytope $P$, we will denote its $n$-dimensional volume by $V_n(P)$. Since $L_T$ is a sublattice of index $2^{|I(T)|}$, we have that $$V_{|E(T)|}(P_T)\cdot |E(T)|!=2^{|I(T)|}\cdot \deg(X_T)$$ providing that the polytope $P_T$ is full-dimensional. Thus, we will concentrate on computing the volumes of the polytopes $P_T$.

Let us consider the unit cube $\fC_T:=[0,1]^{|E(T)|}$.
From the facet description of $P_T$, Theorem~\ref{facetdescription_Z2}, it follows that $$P_T=\fC_T\setminus\bigcup_{v\in I(T),A\in\mathcal{P}_{odd}(\mathcal{N}_T(v))} H^-_{v,A}. $$

The idea of computing the volumes of polytopes $P_T$ is to use the principle of inclusion and exclusion. That means we start with the unit cube with volume one and compute the volumes of the parts that are "cut-off" with hyperplanes  $H_{v,A}$.
A similar approach was used in \cite{RM-claw}. Unfortunately, in the case of arbitrary trees, it will not lead to closed formulas for volumes. However, we will obtain several quite simple recurrence relations. 

In the next result we prove that some parts that are cut-off by hyperplanes $H_{v,A}$ are empty.

\begin{Lemma}\label{0volume}
For any $v\in I(T)$, $A,B\in \mathcal N_T(v)$, the polytope $\fC_T\cap H^-_{v,A}\cap H^-_{v,B}$ is not full-dimensional. Thus, its (lattice) volume is zero.  
\end{Lemma}
\begin{proof}
    The proof is similar to \cite[Lemma 3.3]{RM-claw}.
\end{proof}

\subsection{Polytopes $Q_F$}\label{Q}
In this section, we will compute the volumes of other specific parts that are "cut-off". For this purpose, we define the polytopes $Q_F$ below.

In addition, we will not work only with trees but also with forests. The reason is that later on it will be natural to delete the edge from the tree which will result in a forest. We will use the same notation as for the trees. In particular, for any forest $F$ we still denote $E(F), I(F)$ the set of all edges and inner vertices, respectively.

For any forest $F$ let us denote by $$Q_F:=\fC_F \cap \bigcap_{v\in I(F)} H^-_{v,\emptyset}.$$
Hence, the polytope $Q_F$ is defined by the following inequalities:

\begin{itemize}
\item $0\le x_e\le 1$ for all $e\in E(F)$,
\item For all $v\in I(F)$: $\sum_{e\in\mathcal{N}_F(v)} x_e \le 1.$
\end{itemize}

\begin{Remark}
We can define $Q_F$ even in the pathological case when $F$ has no edges. In that case, $Q_F$ is a "0-dimensional" polytope and, thus, by definition, we have $V_0(Q_F)=1$.
\end{Remark}

The following lemma implies that it would be sufficient to work with trees. However, working with forests will allow us to formulate our results more simply.

\begin{Lemma}\label{P-product-trees}
Let $F$ be a forest with connected components $T_1,\dots, T_k$. Then $Q_{F}=\prod_{i=1}^k Q_{T_i}$. In particular, $$V_{|E(F)|}(Q_{F})=\prod_{i=1}^k V_{|E(T_i)|}(Q_{T_i}).$$
\end{Lemma}
\begin{proof}
 Follows directly from the definition of the polytopes $Q_F$ and $Q_{T_i}$. 
\end{proof}

Note that if $e$ is not the edge from a component isomorphic to $S_1$, the inequalities $x_e\le 1$ are redundant since they can be deduced from the others. Thus, they do not give us facets of $Q_F$.

\begin{Lemma}\label{Vfacets}
Let $P$ be an $n$-dimensional polytope with a point $p\in P$. For a facet $\fF$ of $P$, let us denote $H(\fF)$ the hyperplane defining $\fF$. Then $$V_n(P)=\frac{1}{n}\sum_{\fF} V_{n-1}(\fF)\cdot d(p,H(\fF))$$
where we sum through all facets $\fF$ of $P$.
\end{Lemma}
\begin{proof}
Let us denote by $P_{\fF}$ the convex hull of $\fF$ and $p$. We prove that the polytopes $P_{\fF}$ cover the polytope $P$ and the intersection of any $P_{\fF_1}\cap P_{\fF_2}$ is not full-dimensional.

For the first part, consider any point $x\in P$. The ray from $p$ passing through $x$ intersects the boundary of $P$ in point $y$. Clearly, the point $y$ lies on some facet $\fF_0$ of $P$. However, this implies $x\in P_{\fF_0}$, which proves that the polytopes $P_\fF$ cover the polytope $P$.

For the second part, consider a point $x\in P_{\fF_1}\cap P_{\fF_2}$ for some $\fF_1, \fF_2$ which lies in the interior of $P$. Since $x\in P_{\fF_1}$ there must exist a point $y_1\in \fF_1$ such that $x$ lies on the segment between $p$ and $y_1$. Similarly, there exists a point $y_2\in \fF_2$ such that $x$ lies on the segment between $p$ and $y_2$. However, there is a unique point $y$ in the union of all $\fF_i$ such that $x$ lies on the segment between $y$ and $v$. Thus, $y_1=y_2=y$. Therefore $x\in \conv(\fF_1\cap \fF_2,v)$. Since $\fF_1$ and $\fF_2$ are facets, their intersection is of codimension at least two. Thus, the codimension of $\conv(\fF_1\cap \fF_2,v)$ is at least one which proves that $P_{\fF_1}\cap P_{\fF_2}$ is not full-dimensional.

It follows that $$V_n(P)=\sum_{\fF} V_{n}(P_{\fF}).$$
Since every polytope $P_{\fF}$ is a pyramid over $\fF$, we have $V_n(P_{\fF})=(1/n)\cdot V_{n-1}(\fF)\cdot d(p,H(\fF))$. This proves our result.
\end{proof}

Now we provide formulas for the volume of polytopes $Q_F$.

\begin{Lemma}\label{P-formula}
Let $F$ be a forest and let $S\subset E(F)$ be a set of edges such that $|S\cap \mathcal N_F(v)|=1$ for all $v\in I(F)$. Then $$V_{|E(F)|}(Q_F)=\frac{1}{|E(F)|}\sum_{e\in S} V_{|E(F)|-1}(Q_{F\setminus e}).$$
\end{Lemma}
\begin{proof}
Let us consider a facet $\fF_e$ of $Q_F$ given by the inequality $x_e\ge 0$ for any $e\in E(F)$. We claim that $\fF_e\cong Q_{F\setminus e}$.
Indeed, it is straightforward to check that by putting $x_e=0$ in all inequalities from the facet description of $Q_F$, we obtain the facet description of $Q_{F\setminus e}$.

Let us denote by $x(S)$ the point with coordinates $x(S)_e=1$ for $e\in S$ and $x(S)_e=0$ for $e\not\in S$. We can easily check that $x(S)\in Q_F$; in fact, $x(S)$ is a vertex of $Q_F$. It is easy to check that the point $x(S)$ lies on many facets of $Q_F$. The only facets of $Q_F$ to which it does not belong are the facets $\fF_e$ for $e\in S$. Now we apply Lemma~\ref{Vfacets} for the polytope $Q_F$ and point $x(S)$ to obtain:

$$V_{|E(F)|}(Q_F)=\frac{1}{|E(F)|}\sum_{e\in S} V_{|E(F)|-1}(\fF_e)\cdot d(x(S),H(\fF_e))=\frac{1}{|E(F)|}\sum_{e\in S} V_{|E(F)|-1}(Q_{F\setminus e})\cdot 1.$$

Note that it was sufficient to sum up only through facets $\fF_e$ for $e\in S$ since the point $x(S)$ has distance $0$ from the others. Clearly $d(x(S),\fF_e)=1$ for $e\in S$. This concludes the lemma.

\end{proof}
\begin{Corollary}\label{P-star}
$V_n(Q_{S_n})=1/n!$.
\end{Corollary}
\begin{proof}
The statement is true for $n=1$, and then one can proceed by induction using Lemma \ref{P-formula}. Alternatively, it can be checked directly, since $Q_{S_n}$ is a unit simplex.
\end{proof}
\begin{Corollary}\label{P-doublestar}
$$V_{m+n+1}(Q_{S_{m,n}})=\frac{1}{m!n!(m+n+1)}.$$
\end{Corollary}
\begin{proof}
By using the formula from Lemma~\ref{P-formula} for the set $S=\{e\}$ where $e$ is the edge connecting two inner vertices, we obtain $$V_{m+n+1}(Q_{S_{m,n}})=\frac{1}{m+n+1}\cdot V_{m+n}(Q_{S_{m,n}\setminus e})=$$
$$=\frac{1}{m+n+1}ˇ\cdot V_m(Q_{S_m})\cdot V_n(Q_{S_n})=\frac{1}{m+n+1}\cdot\frac{1}{m!}\cdot \frac{1}{n!}.$$

We used Lemma \ref{P-product-trees} for the graph $S_{m,n}\setminus e$ which has two components that are $S_m$ and $S_n$.
\end{proof}

\begin{Example}\label{example-P}
We illustrate the formula presented in Lemma \ref{P-formula} for the tree $T$ shown in Figure~\ref{71}:

\vspace{.1in}
\begin{figure}
\centering
\begin{tikzpicture}
\draw [black, thick] (2,1)--(1,2);
\draw [black, thick] (1,0)--(2,1);
\draw [black, thick, red] (2,1)--(4,1);
\node [right, red, ultra thick] at (2.6,1.3) {$e_1$};
\draw [black, thick] (4,1)--(4,2);
\draw [black, thick] (4,1)--(6,1)--(7,0);
\draw [black, thick,red] (7,2)--(6,1);
\node [right, black, ultra thick] at (5.8,0.3) {$e_3$};
\node [right, red, ultra thick] at (5.8,1.7) {$e_2$};
\end{tikzpicture}
\caption{}
\label{71}
\end{figure}

Let $S=\{e_1,e_2\}$ be the set consisting of two red edges on the picture. Then 
$$V_7(Q_T)=\frac 17\left(V_6(Q_{S_2\cup S_{1,2}})+V_6(Q_{T\setminus e_2}) \right).$$

For the union of star and double star we can use Lemma \ref{P-product-trees} and Corollaries \ref{P-star} and \ref{P-doublestar} to obtain:

$$V_6(Q_{S_2\cup S_{1,2}})=V_2(Q_{S_2})\cdot V_4(Q_{S_{2,3}})=\frac{1}{2!}\cdot \frac{1}{1!\cdot2!\cdot 4}=\frac{1}{16}.$$

For $T\setminus e_2$ we use again Lemma~\ref{P-formula}, this time taking $S=\{e_1,e_3\}$:

$$V_6(Q_{T\setminus e_2})=\frac{1}{6}\left(V_5(Q_{S_2\cup S_{1,1}})+V_5(Q_{S_{2,2}})\right)=\frac{1}{6}\left(\frac{1}{2!}\cdot \frac{1}{1!\cdot1!\cdot3}+\frac{1}{2!\cdot 2!\cdot 5} \right)=
\frac{1}{6}\cdot\frac{13}{60}=\frac{13}{360}.$$

Thus,
$$V_7(Q_T)=\frac 17\left(\frac{1}{16}+\frac{13}{360} \right)
=\frac{71}{7!}.$$
\end{Example}

\begin{Remark}
Instead of computing volumes of polytopes $Q_F$ we can compute their lattice volumes in the lattice $\ZZ^{|E(F)|}$. The advantage is that the results would be integers (for instance, in the previous example, we would get 71). However, some of the formulas are easier to state with volumes, others are easier to state with lattice volumes. It is not clear which is more convenient to work with. Here we decided to work with volumes.
\end{Remark}

Note that the formula from Lemma~\ref{P-formula} together with the fact $V_1(Q_{S_1})=1$ allows us to recursively compute the volume of $Q_F$ for any forest $F$. 

\begin{Lemma}\label{spliting}
Let $F$ be a forest. Assume that there exist two adjacent inner vertices $v_1,v_2$ for which there are edges $e_1\in \mathcal N_F(v_1)$, $e_2\in \mathcal N_F(v_2)$ such that one endpoint of both $e_1,e_2$ is a leaf. Let us denote by $e_0$ the edge between $v_1$ and $v_2$. Then $$
V_{|E(F)|-1}(Q_{T\setminus e_1})+V_{|E(F)|-1}(Q_{T\setminus e_2})=V_{|E(F)|-1}(Q_{T\setminus e_0}).$$
\end{Lemma}
\begin{proof}
Let $S$ be a set of edges of $F$ such that $|S\cap \mathcal N_F(v)|=1$ for every $v\in I(F)$ and $e_0\in S$. Such a set $S$ always exists. Notice that the set $S':=S\setminus\{e_0\}\cup\{e_1,e_2\}$ also has the property that $|S'\cap \mathcal N_F(v)|=1$ for every $v\in I(F)$. Thus, we can use Lemma~\ref{P-formula} to express the volume of $Q_T$ for both sets $S$ and $S'$. By comparing these two expressions we obtain the statement of our result. 
\end{proof}

For a forest $F$ with a vertex $v$ and a positive integer $k$, let us denote by $[F,v,k]$ the forest obtained by adding $k$ leaves to the vertex $v$.

\begin{Corollary}\label{P-splitting-formula}
Let $F$ be a forest with an inner edge $e_0$ connecting the vertices $v$ and $w$. Let $e_1$ be an edge that is adjacent to $v$ such that its other end is a leaf. Then 

$$
V_{|E(F)|}(Q_{F})=-V_{|E(F)|}(Q_{[T\setminus e_1,w,1]})+V_{|E(F)|}(Q_{[T\setminus e_0,w,1]}).$$
\end{Corollary}
\begin{proof}
We apply Lemma~\ref{spliting} for the forest $[T,w,1]$.   
\end{proof}

For a better understanding of the relation from Corollary \ref{P-splitting-formula}, we refer to Figure~\ref{Q_T}:

\begin{figure}[H]
\begin{tikzpicture}
\draw [black, thick] (1,0)--(2,1)--(1,2);
\draw [black, thick] (2,1)--(4,1);
\draw [black, thick] (2,1)--(0.5,1.5);
\draw [black, thick] (2,1)--(0.5,0.5);
\draw [black, thick] (0,1)--(2,1);
\draw [red, thick] (2,1)--(2,2);
\draw [black, thick] (4,1)--(5,2)--(4,1)--(6,1)--(4,1)--(5,0);
\draw [black, thick] (4,1)--(5.5,1.5);
\draw [black, thick] (4,1)--(5.5,0.5);
\node [right, black, ultra thick] at (0.5,1.7) {$T_2$};
\node [right, black, ultra thick] at (1,2) {$T_1$};
\node [right, black, ultra thick] at (0,1.2) {$T_3$};
\node [right, black, ultra thick] at (-0.1,0.4) {$...$};
\node [right, black, ultra thick] at (1,0) {$T_k$};

\node [right, black, ultra thick] at (5.5,0.5) {$....$};
\node [right, black, ultra thick] at (5.5,1.5) {$T_{k+2}$};
\node [right, black, ultra thick] at (5,2) {$T_{k+1}$};
\node [right, black, ultra thick] at (6,1) {$T_{k+3}$};
\node [right, black, ultra thick] at (5, 0) {$T_{l}$};
\node [right, blue, ultra thick] at (7,1) {\huge $=$};

\end{tikzpicture}
\end{figure}

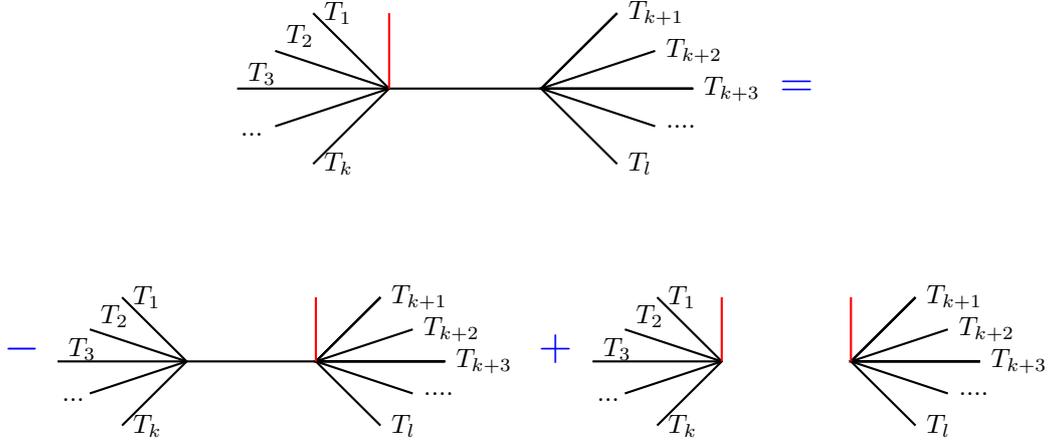
\begin{figure}[H]
\begin{minipage}[c]{0.5\linewidth}
\begin{tikzpicture}[scale=0.85]
\centering
\draw [black, thick] (1,0)--(2,1)--(1,2);
\draw [black, thick] (2,1)--(4,1);
\draw [black, thick] (2,1)--(0.5,1.5);
\draw [black, thick] (2,1)--(0.5,0.5);
\draw [black, thick] (0,1)--(2,1);
\draw [black, thick] (4,1)--(5,2)--(4,1)--(6,1)--(4,1)--(5,0);
\draw [black, thick] (4,1)--(5.5,1.5);
\draw [black, thick] (4,1)--(5.5,0.5);
\node [right, black, ultra thick] at (0.5,1.7) {$T_2$};
\node [right, black, ultra thick] at (1,2) {$T_1$};
\node [right, black, ultra thick] at (0,1.2) {$T_3$};
\node [right, black, ultra thick] at (-0.1,0.4) {$...$};
\node [right, black, ultra thick] at (1,0) {$T_k$};
\draw [red, thick] (4,1)--(4,2);
\node [right, black, ultra thick] at (5.5,0.5) {$....$};
\node [right, black, ultra thick] at (5.5,1.5) {$T_{k+2}$};
\node [right, black, ultra thick] at (5,2) {$T_{k+1}$};
\node [right, black, ultra thick] at (6,1) {$T_{k+3}$};
\node [right, black, ultra thick] at (5, 0) {$T_{l}$};
\node [right, blue, ultra thick] at (-1,1.2) {\huge $-$};
\end{tikzpicture}
\end{minipage}\hfill
\begin{minipage}[c]{0.5\linewidth}
\centering
\begin{tikzpicture}[scale=0.85]
\draw [black, thick] (1,0)--(2,1)--(1,2);
\draw [black, thick] (2,1)--(0.5,1.5);
\draw [black, thick] (2,1)--(0.5,0.5);
\draw [black, thick] (0,1)--(2,1);
\draw [red, thick] (2,1)--(2,2);
\draw [black, thick] (4,1)--(5,2)--(4,1)--(6,1)--(4,1)--(5,0);
\draw [black, thick] (4,1)--(5.5,1.5);
\draw [black, thick] (4,1)--(5.5,0.5);
\node [right, black, ultra thick] at (0.5,1.7) {$T_2$};
\node [right, black, ultra thick] at (1,2) {$T_1$};
\node [right, black, ultra thick] at (0,1.2) {$T_3$};
\node [right, black, ultra thick] at (-0.1,0.4) {$...$};
\node [right, black, ultra thick] at (1,0) {$T_k$};
\draw [red, thick] (4,1)--(4,2);
\node [right, black, ultra thick] at (5.5,0.5) {$....$};
\node [right, black, ultra thick] at (5.5,1.5) {$T_{k+2}$};
\node [right, black, ultra thick] at (5,2) {$T_{k+1}$};
\node [right, black, ultra thick] at (6,1) {$T_{k+3}$};
\node [right, black, ultra thick] at (5, 0) {$T_{l}$};
\node [right, blue, ultra thick] at (-1,1.2) {\huge $+$};
\end{tikzpicture}
\end{minipage}
\caption{Recursive formula for volumes of polytopes $Q_F$}\label{Q_T}
\end{figure}

 Corollary~\ref{P-splitting-formula} also allows us to recursively compute $V_{|E(T)|}(Q_T)$. Even though both graphs on the right-hand side have the same number of edges, by systematically moving the leaves toward one vertex we will end up with a forest in which every component is a star.

\subsection{Volume}\label{vol}
For a forest $F$ with an inner edge $e$, let us define a graph $G//e:=[F/e,v,1]$ where $v$ is the contracted vertex. In other words, we contract the edge but we keep it as a leaf. Thus, $|E(F//e)|=|E(F)|$.

\begin{Example}

Consider a tree $T$ with three inner vertices of degrees 4,3 and 3. If $e$ is the inner edge between the first two inner vertices, then the tree $T//e$ has two inner vertices of degrees 6 and 3 as in Figure~\ref{example//e}.

\begin{figure}[H]
    
\begin{minipage}[c]{0.5\linewidth}
\centering
\begin{tikzpicture}
\draw [black, thick] (0.3,1)--(2,1)--(1,2);
\draw [black, thick] (1,0)--(2,1);
\draw [black, thick, red] (2,1)--(4,1);
\draw [black, thick] (4,1)--(4,2);
\draw [black, thick] (4,1)--(6,1)--(7,2)--(6,1)--(7,0);
\end{tikzpicture}
\caption*{$T$}
\end{minipage}\hfill
\begin{minipage}[c]{0.4\linewidth}
\centering
\begin{tikzpicture}
\draw [black, thick] (0.3,1)--(2,1)--(1,2);
\draw [black, thick] (1,0)--(2,1);
\draw [red, thick] (2,1)--(2,2);
\draw [black, thick] (2,1)--(3,2);
\draw [black, thick] (2,1)--(4,1);
\draw [black, thick] (4,1)--(5,2)--(4,1)--(5,0);
\end{tikzpicture}
\caption*{$T//e$}
\end{minipage}
\caption{}\label{example//e}
\end{figure}

\end{Example}
For a set $E':=\{e_1,\dots,e_k\}\subset E(F)$ we define the graph $$F//E':=((\dots(F//e_1)//e_2)//\dots//e_k).$$
Note that the resulting graph does not depend on the order in which we are contracting the edges.

For a set $I':=\{v_1,\dots,v_k\}\subset I(F)$, we define the forest $G(I')$ as the forest with inner vertices $I'$ and the edges $\bigcup_{i=1}^k \mathcal N_F(v_i)$ such that $\mathcal N_{G(I')}(v_i)=\mathcal N_F(v_i)$.

\begin{Example}
Consider the tree $T$ with three inner vertices of degrees 3, 4, and 4. Let the set $I'$ be formed by the first and last inner vertex. Then $G(I')$ is a forest that consists of 3-star and 4-star, as in Figure~\ref{exampleG(I)}.
\begin{figure}[H]
\begin{minipage}[c]{0.5\linewidth}
\centering
\begin{tikzpicture}
\foreach \x in {2,6}{
\filldraw [red] (\x,1) circle (3pt);}
\draw [black, thick] (1,0)--(2,1)--(1,2);
\draw [black, thick] (2,1)--(4,1)--(3,2)--(4,1)--(5,2);
\draw [black, thick] (4,1)--(6,1)--(7,2)--(6,1)--(7,1)--(6,1)--(7,0);
\end{tikzpicture}
\caption*{$T$ with marked set $I'$}
\end{minipage}\hfill
\begin{minipage}[c]{0.5\linewidth}
\centering
\begin{tikzpicture}

\foreach \x in {2,6}{
\filldraw [red] (\x,1) circle (3pt);}
\draw [black, thick] (1,0)--(2,1)--(1,2);
\draw [black, thick] (2,1)--(3.8,1);
\draw [black, thick] (4.2,1)--(6,1)--(7,2)--(6,1)--(7,1)--(6,1)--(7,0);
\end{tikzpicture}
\caption*{$G(I')$}
\end{minipage}
\caption{}
\label{exampleG(I)}
\end{figure}

\end{Example}

\begin{Lemma}\label{beautiful}
Let $F$ be a forest and $I':=\{v_1,\dots, v_k\}\subset I(F)$. Denote by $l$ the number of edges in $\bigcup_{i=1}^k \mathcal N_F(v_i)$. Let $A_1,\dots,A_k$ be sets of edges of $F$ such that $A_i\subset \mathcal N_F(v_i)$. Consider the set $$E':=\{e\in IE(F): \exists \ i,j: e\in A_i\cap \mathcal N_F(v_j), e\not\in A_j \}.$$ Then $$\fC_F\cap\bigcap_{i=1}^k H^-_{v_i,A_i}\cong Q_{G(I')//E'}\times[0,1]^{|E(F)|-l}.$$
In particular, $$V_{|E(F)|}\left(\fC_F\cap\bigcap_{i=1}^k H^-_{v_i,A_i}\right)=V_l (Q_{G(I')//E'}).$$
\end{Lemma}
\begin{proof}
We can consider the polytope $Q_{G(I')//E'}\times[0,1]^{|E(F)|-l}$ to be naturally embedded in $\RR^{|E(F)|}$ by simply keeping the coordinates corresponding to edges that are in $G(I')$ and allowing all other coordinates to be any number from $[0,1]$.

Let $\varphi: E'\rightarrow [k]$ be an injective function such that for all $e\in E'$ we have $e\in \mathcal N_F(v_{\varphi(e)})$. Since such a function exists, then one can, for example, make the forest rooted and direct the edge away from the roots. Then $\varphi(e)$ assigns to every edge the index of the vertex to which $e$ is pointing. 

We will make a linear change of coordinates which will transform the polytope $\fC_F\cap\bigcap_{i=1}^k H^-_{v_i, A_i}$ to the polytope $Q_{G(I')//E'}\times[0,1]^{|E(F)|-l}$. 

The change of coordinates is as follows:
\begin{itemize}
    \item $x'_e=x_e$ $\forall e\not\in\bigcup_{i=1}^k A_i$,
    \item $x'_e=1-x_e$ $\forall e\in(\bigcup_{i=1}^k A_i)\setminus E'$,
    \item $x'_e=1-|A_{\varphi(e)}|-S_{v_{\varphi(e)},A_{\varphi(e)}}(x)$ $\forall e\in E'$.
\end{itemize}
 It is easy to see that this change of coordinates gives us the lattice isomorphism of $\ZZ^{|E(F)|}$. Thus, it preserves the volumes of polytopes.

It remains to verify that the image of the polytope $\fC_F\cap\bigcap_{i=1}^k H^-_{v_i,A_i}$ is the polytope $Q_{G(I')//E'}\times[0,1]^{|E(F)|-l}$. It is sufficient to check that a point $x\in \RR^{|E(F)|}$ satisfies all the inequalities from the facet description of $\fC_F\cap\bigcap_{i=1}^k H^-_{v_i,A_i}$ if and only if its image, the point $x'$, satisfies all inequalities from the facet description of $Q_{G(I')//E'}\times[0,1]^{|E(F)|-l}$.

Let us list all the inequalities from the facet description of both polytopes. For the intersection
$\fC_F\cap\bigcap_{i=1}^k H^-_{v_i,A_i}$:

\begin{itemize}
\item $0\le x_e\leq 1$ for all $e\in E(F)$,
\item For all $1\le i \le k$:
\[
\sum_{e\in\mathcal{N}_F(v_i)\setminus A_i} x_e - \sum_{e\in A_i} x_e\le 1-|A_i|.\]
\end{itemize}

The inequalities for $Q_{G(I')//E'}\times[0,1]^{|E(F)|-l}$: 
\begin{itemize}
\item $0\le x'_e\leq 1$ for all $e\in E(F)$,
\item For all $v\in I(G(I')//E')$:
\[
\sum_{e\in\mathcal{N}_{G(I')//E'}(v)} x_e\le 1.\]
\end{itemize}

Suppose that $x\in \fC_F\cap\bigcap_{i=1}^k H^-_{v_i,A_i}$. Clearly, we have $0\le x'_e\le 1$ for all $e\not\in E'$. Consider an edge $e\in E'$. Since $x\in H^-_{v_{\varphi(e)},A_{\varphi(e)}}$ we have $$
x'_e=1-|A_{\varphi(e)}|-S_{v_{\varphi(e)},A_{\varphi(e)}}(x)\ge 0.$$

On the other hand, $$x'_e=1-|A_{\varphi(e)}|-S_{v_{\varphi(e)},A_{\varphi(e)}}(x)=1-\sum_{e\in\mathcal{N}_F(v_i)\setminus A_i} x_e - \sum_{e\in A_i}(1-x_e)\le 1.$$

This proves that also for $e\in E'$, we have $0\le x'_e\le 1$. Now let us consider an inner vertex $v\in I(G(I')//E')$. From the definition of the graph $G(I'//E')$ there exist vertices $v_{i_0},\dots, v_{i_m}\in I', m\ge 0$, such that $v$ is a contraction of those vertices. This implies that all edges in the graph induced by $v_{i_0},\dots,v_{i_m}$ are in $E'$, and there are exactly $m$ such edges. Thus, there exists exactly one index $j$ such that there exists no edge $e\in E'$ for which $\varphi(e)=v_{i_j}$. Without loss of generality, this index is $j=0$. By definition, we have $\mathcal N_{G(I'//E')}(v)=\bigcup_{j=0}^m \mathcal N_F (v_{i_j})$.   Therefore, we obtain:

$$\sum_{e\in\mathcal{N}_{G(I')//E'}(v)} x'_e=\sum_{e\in\bigcup_{j=0}^m \mathcal N_F (v_{i_j})\setminus A_{i_j}} x_e+\sum_{e\in(\bigcup_{j=0}^m A_{i_j})\setminus E'} (1-x_e)+$$
$$+\sum_{e\in E'\cap \mathcal N_{G(I')//E'}(v)} \left(1-|A_{\varphi(e)}|-S_{v_{\varphi(e)},A_{\varphi(e)}}(x)\right)=  $$

$$=\sum_{j=0}^m\left(\sum_{e\in \mathcal N_F (v_{i_j})\setminus (A_{i_j}\cup E')} x_e+\sum_{e\in A_{i_j}\setminus E'} (1-x_e)\right)+$$
$$+\sum_{e\in E'\cap \mathcal N_{G(I')//E'}(v)} \left(x_e+(1-x_e)-|A_{\varphi(e)}|-S_{v_{\varphi(e)},A_{\varphi(e)}}(x)\right)=$$

$$=\sum_{j=0}^m\left(\sum_{e\in \mathcal N_F (v_{i_j})\setminus A_{i_j}} x_e+\sum_{e\in A_{i_j}} (1-x_e)\right)+$$
$$+\sum_{e\in E'\cap \mathcal N_{G(I')//E'}(v)} \left(-|A_{\varphi(e)}|-S_{v_{\varphi(e)},A_{\varphi(e)}}(x)\right)=
$$

$$=\sum_{j=0}^m\left( S_{v_{i_j},A_{i_j}}(x)+|A_{i_j}|\right)+\sum_{j=1}^{m} \left(-|A_{v_{i_j}}|-S_{v_{i_j},A_{i_j}}(x)\right)=
$$
$$=S_{v_{i_0},A_{i_0}}+|A_{v_{i_0}}|\le 1-|A_{v_{i_0}}|+|A_{v_{i_0}}|=1.$$

This shows that $x'\in Q_{G(I')//E'}\times[0,1]^{|E(F)|-l} $. Now we consider any $x'\in Q_{G(I')//E'}\times[0,1]^{|E(F)|-l} $. Clearly $0\le x_e\le 1$ for all $e\not\in E'$. Next, let us consider a vertex $v_{i_0}$, such that there exists an edge $e_0\in E'$ for which $v_{i_0}=\varphi(e_0)$. Then

$$S_{v_{i_0}, A_{i_0}}(x)=1-|A_{i_0}|-x'_e\le 1-|A_{i_0}|,$$
which shows that the inequality holds for the pair $(v_{i_0}, A_{i_0})$. Consider now a vertex $v_{i_0}$ such that there is no edge $e\in E'$ for which $v_i=\varphi(e)$.  In the graph $G(I')//E'$, there exists a vertex $v$ such that $v$ is a contraction of vertices $v_{i_0},\dots, v_{i_m}$. As in the previous argument, we have

$$\sum_{e\in\mathcal{N}_{G(I')//E'}(v)} x'_e=S_{v_{i_0},A_{i_0}}+|A_{v_{i_0}}|.$$

This implies that $$S_{v_{i_0},A_{i_0}}=\sum_{e\in\mathcal{N}_{G(I')//E'}(v)} x'_e-|A_{v_{i_0}}|\le 1-|A_{v_{i_0}}|.$$

It remains to note that the inequalities $0\le x_e\le 1$ for $e\in E'$ are redundant in the description of $\fC_F\cap\bigcap_{i=1}^k H^-_{v_i,A_i}$, i.e. they do not give us facets. To prove this, assume that for a point $x$ all the other inequalities from the facet description of $\fC_F\cap\bigcap_{i=1}^k H^-_{v_i,A_i}$ hold, but there exists an edge $e\in E'$ such that $x_{e}\not\in [0,1]$.

Consider any such edge $e_0$ for which $x_{e_0}>1$. Then there exists an index $j$ such that the vertex $v_j$ is adjacent to $e_0$ and $e_0\not\in A_j$. Thus, $$\sum_{e\in\mathcal{N}_F(v_j)\setminus A_j} x_e - \sum_{e\in A_j} x_e\le 1-|A_j|\Leftrightarrow \sum_{e\in (\mathcal{N}_F(v_j)\setminus A_j\cup\{e_0\})} x_e + \sum_{e\in A_j} (1-x_e)\le 1-x_{e_0}. $$ 

Therefore, there must exist another edge $e_1$ such that either $x_{e_1}<0$ and $e_1\not\in A_j$ or $x_{e_1}>1$ and $e_1\in A_j$.

Similarly, if we have an edge $e_0$ with $x_{e_0}<0$, we find an index $j$ such that the vertex $v_j$ is adjacent to $e_0$ and $e_0\in A_j$. Analogously, we get the conclusion that there exists an edge $e_1\in \mathcal N_F (v_j)$ such that either $x_{e_1}<0$ and $e_1\not\in A_j$ or $x_{e_1}>1$ and $e_1\in A_j$.

We see that starting in some edge $e_0$, we can create a path of edges $e_0,e_1,\dots$ such that for all of these edges, we have $x_{e_i}\not\in [0,1]$. Eventually, we must get to the situation where one of these edges is not in $E'$, which is a contradiction.

This also shows that, for a point $x'\in Q_{G(I')//E'}\times[0,1]^{|E(F)|-l}$, we have $x\in \fC_F\cap\bigcap_{i=1}^k H^-_{v_i,A_i}$ which finishes the proof.

\end{proof}

For a forest $F$, let us denote $$r(F):=\sum_{E'\subset IE(F)}V_{|E(F)|}(Q_{F//E'}).$$
\begin{Lemma}\label{Q-product-trees}
 Let $F$ be a forest with connected components $T_1,\dots, T_k$. Then $$r(F)=\prod_{i=1}^k r(T_i).$$
\end{Lemma}
\begin{proof}
$$r(F)=\sum_{E'\subset IE(F)}V_{|E(F)|}(Q_{F//E'})=\sum_{E'_1\subset IE(T_1)}\dots\sum_{E'_k\subset IE(T_k)}V_{|E(F)|}(Q_{F//(E'_1\cup\dots\cup E'_k}))=$$
$$=\sum_{E'_1\subset IE(T_1)}\dots\sum_{E'_k\subset IE(T_k)}V_{|E(F)|}(Q_{(T_1//E'_1)\oplus\dots\oplus (T_k//E'_k)})=$$
$$=\sum_{E'_1\subset IE(T_1)}\dots\sum_{E'_k\subset IE(T_k)}\prod_{i=1}^kV_{|E(T_i)|}(Q_{T_i//E'_i})=\prod_{i=1}^k\sum_{E'_i\subset IE(T_k)}V_{|E(T_i)|}(Q_{T_i//E'_i})=\prod_{i=1}^k r(T_i).$$
We use Lemma \ref{P-product-trees} for the volume of $Q_{(T_1//E'1)\oplus\dots\oplus (T_k//E'_k)}$, where $\oplus$ stands for a disjoint union of graphs.
\end{proof}

\begin{Lemma}\label{qstar}
For all $m,n\ge 1$ we have 
\begin{enumerate}[label=(\alph*)]
    \item $r(S_n)=1/n!.$
    \item  $r(S_{m,n})=\frac{1}{(m+n+1)m!n!}+\frac{1}{(m+n+1)!}.$
\end{enumerate}
\end{Lemma}
\begin{proof}
We use the definition of $r(F)$. Since $S_n$ has no inner edges, we have $r(S_n)=V_n(Q_{S_n})=1/n!$.

The double star $S_{m,n}$ has one inner edge, hence, we obtain
$$r(S_{m,n})=V_{m+n+1}(Q_{S_{m,n}})+V_{m+n+1}(Q_{S_{m+n+1}})=\frac{1}{(m+n+1)m!n!}+\frac{1}{(m+n+1)!}.$$
We use Corollary \ref{P-star} and \ref{P-doublestar} for the volumes of polytopes $Q_{S_{m+n+1}}$ and $Q_{S_{m,n}}$.
\end{proof}

\begin{Example}\label{r(T)}
We illustrate the computation of $r(T)$ for the tree $T$ presented in Example~\ref{example-P}. We refer now to Figure~\ref{102}.
\vspace{.1in}
\begin{figure}[H]
\centering
\begin{tikzpicture}
\draw [black, thick] (2,1)--(1,2);
\draw [black, thick] (1,0)--(2,1);
\draw [black, thick, red] (2,1)--(4,1);
\node [right, red, ultra thick] at (2.6,1.3) {$e_1$};
\draw [black, thick] (4,1)--(4,2);
\draw [red, thick] (4,1)--(6,1);
\draw [black, thick] (6,1)--(7,0);
\node [right, red, ultra thick] at (4.6,1.3) {$e_4$};
\draw [black, thick] (7,2)--(6,1);
\end{tikzpicture}
\caption{}
\label{102}
\end{figure}
The set of inner edges $IE(T)=\{e_1,e_4\}$ has four subsets. Thus
$$r(T)=V_7(Q_T)+V_7(Q_{S_{4,2}})+V_7(Q_{S_{2,4}})+V_7(Q_{S_7})=
\frac{71}{7!}+\frac{2}{2!\cdot4!\cdot 7}+\frac{1}{7!}=\frac{102}{7!}.$$

\end{Example}

\begin{Lemma}\label{Q-splitting-formula}
Let $F$ be a forest with an inner edge $e_0$ connecting the vertices $v$ and $w$. Let $e_1$ be an edge adjacent to $v$ such that its other end is a leaf. Assume that the degree of $v$ is at least 3. Then the following formula holds (see also Figure~\ref{picture-Qsplit}):

$$
r(F)=-r([F\setminus e_1,w,1])+r([F\setminus e_0,w,1])+2r(F//e_0).$$

\begin{figure}[H]
\begin{minipage}[c]{0.5\linewidth}
\centering
\begin{tikzpicture}[scale=0.85]
\draw [black, thick] (1,0)--(2,1)--(1,2);
\draw [black, thick] (2,1)--(4,1);
\draw [black, thick] (2,1)--(0.5,1.5);
\draw [black, thick] (2,1)--(0.5,0.5);
\draw [black, thick] (0,1)--(2,1);
\draw [red, thick] (2,1)--(2,2);
\draw [black, thick] (4,1)--(5,2)--(4,1)--(6,1)--(4,1)--(5,0);
\draw [black, thick] (4,1)--(5.5,1.5);
\draw [black, thick] (4,1)--(5.5,0.5);
\node [right, black, ultra thick] at (0.5,1.7) {$T_2$};
\node [right, black, ultra thick] at (1,2) {$T_1$};
\node [right, black, ultra thick] at (0,1.2) {$T_3$};
\node [right, black, ultra thick] at (-0.1,0.4) {$...$};
\node [right, black, ultra thick] at (1,0) {$T_k$};

\node [right, black, ultra thick] at (5.5,0.5) {$....$};
\node [right, black, ultra thick] at (5.5,1.5) {$T_{k+2}$};
\node [right, black, ultra thick] at (5,2) {$T_{k+1}$};
\node [right, black, ultra thick] at (6,1) {$T_{k+3}$};
\node [right, black, ultra thick] at (5, 0) {$T_{l}$};
\node [right, blue, ultra thick] at (7,1) {\huge $=$};

\end{tikzpicture}
\end{minipage}\hfill
\begin{minipage}[c]{0.5\linewidth}
\centering
\begin{tikzpicture}[scale=0.85]

\draw [black, thick] (1,0)--(2,1)--(1,2);
\draw [black, thick] (2,1)--(4,1);
\draw [black, thick] (2,1)--(0.5,1.5);
\draw [black, thick] (2,1)--(0.5,0.5);
\draw [black, thick] (0,1)--(2,1);
\draw [black, thick] (4,1)--(5,2)--(4,1)--(6,1)--(4,1)--(5,0);
\draw [black, thick] (4,1)--(5.5,1.5);
\draw [black, thick] (4,1)--(5.5,0.5);
\node [right, black, ultra thick] at (0.5,1.7) {$T_2$};
\node [right, black, ultra thick] at (1,2) {$T_1$};
\node [right, black, ultra thick] at (0,1.2) {$T_3$};
\node [right, black, ultra thick] at (-0.1,0.4) {$...$};
\node [right, black, ultra thick] at (1,0) {$T_k$};
\draw [red, thick] (4,1)--(4,2);
\node [right, black, ultra thick] at (5.5,0.5) {$....$};
\node [right, black, ultra thick] at (5.5,1.5) {$T_{k+2}$};
\node [right, black, ultra thick] at (5,2) {$T_{k+1}$};
\node [right, black, ultra thick] at (6,1) {$T_{k+3}$};
\node [right, black, ultra thick] at (5, 0) {$T_{l}$};
\node [right, blue, ultra thick] at (-1,1) {\huge $-$};
\end{tikzpicture}
\end{minipage}
\end{figure}
\begin{figure}[H]
\begin{minipage}[c]{0.5\linewidth}
\centering
\begin{tikzpicture}[scale=0.85]
\draw [black, thick] (1,0)--(2,1)--(1,2);
\draw [black, thick] (2,1)--(0.5,1.5);
\draw [black, thick] (2,1)--(0.5,0.5);
\draw [black, thick] (0,1)--(2,1);
\draw [red, thick] (2,1)--(2,2);
\draw [black, thick] (4,1)--(5,2)--(4,1)--(6,1)--(4,1)--(5,0);
\draw [black, thick] (4,1)--(5.5,1.5);
\draw [black, thick] (4,1)--(5.5,0.5);
\node [right, black, ultra thick] at (0.5,1.7) {$T_2$};
\node [right, black, ultra thick] at (1,2) {$T_1$};
\node [right, black, ultra thick] at (0,1.2) {$T_3$};
\node [right, black, ultra thick] at (-0.1,0.4) {$...$};
\node [right, black, ultra thick] at (1,0) {$T_k$};
\draw [red, thick] (4,1)--(4,2);
\node [right, black, ultra thick] at (5.5,0.5) {$....$};
\node [right, black, ultra thick] at (5.5,1.5) {$T_{k+2}$};
\node [right, black, ultra thick] at (5,2) {$T_{k+1}$};
\node [right, black, ultra thick] at (6,1) {$T_{k+3}$};
\node [right, black, ultra thick] at (5, 0) {$T_{l}$};
\node [right, blue, ultra thick] at (-1,1.1) {\huge $+$};
\end{tikzpicture}
\end{minipage}\hfill
\begin{minipage}[c]{0.5\linewidth}
\centering
\begin{tikzpicture}
\draw [black, thick] (1,0)--(2,1)--(1,2);
\draw [black, thick] (2,1)--(0.5,1.5);
\draw [black, thick] (2,1)--(0.5,0.5);
\draw [black, thick] (0,1)--(2,1);
\draw [black, thick] (2,1)--(3,2)--(2,1)--(4,1)--(2,1)--(3,0);
\draw [black, thick] (2,1)--(3.5,1.5);
\draw [black, thick] (2,1)--(3.5,0.5);
\node [right, black, ultra thick] at (0.5,1.7) {$T_2$};
\node [right, black, ultra thick] at (1,2) {$T_1$};
\node [right, black, ultra thick] at (0,1.2) {$T_3$};
\node [right, black, ultra thick] at (-0.1,0.4) {$...$};
\node [right, black, ultra thick] at (1,0) {$T_k$};
\draw [red, thick] (2,1)--(2.4,2);
\draw [red, thick] (2,1)--(1.6,2);
\node [right, black, ultra thick] at (3.5,0.5) {$....$};
\node [right, black, ultra thick] at (3.5,1.5) {$T_{k+2}$};
\node [right, black, ultra thick] at (3,2) {$T_{k+1}$};
\node [right, black, ultra thick] at (4,1) {$T_{k+3}$};
\node [right, black, ultra thick] at (3, 0) {$T_{l}$};
\node [right, blue, ultra thick] at (-1.2,1.1) {\huge $+2\cdot$};
\end{tikzpicture}
\end{minipage}
\caption{Recursive formula for $r(F)$}
\label{picture-Qsplit}
\end{figure}
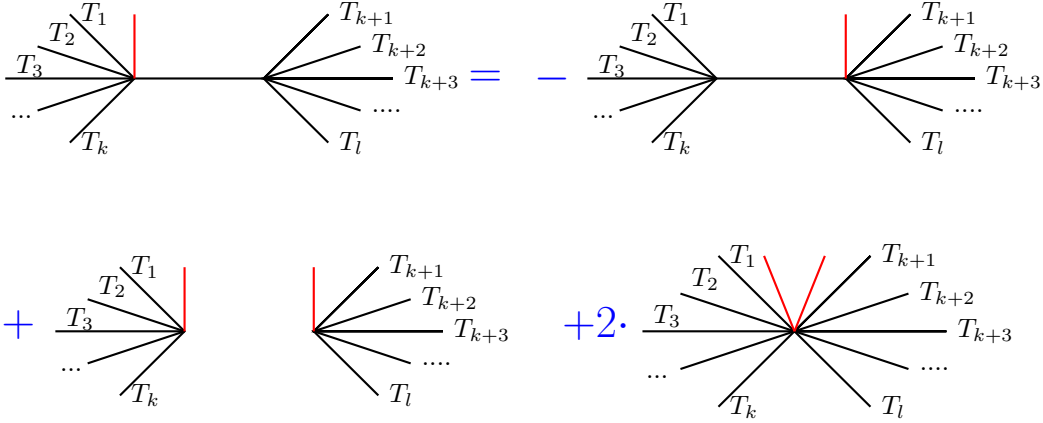
\end{Lemma}
\begin{proof}
In the definition of $r(Q_F)$ we split the sum to those sets $E'$ that contain edge $e_0$ and those sets that do not:

$$r(F)=\sum_{E'\subset IE(F)\setminus\{e_0\}}V_{|E(F)|}(Q_{F//E'})+\sum_{E'\subset IE(F), e_0\in E'}V_{|E(F)|}(Q_{F//E'}).$$

First, we look at the first term. For sets $E'$ that do not contain $e_0$, by Lemma~\ref{P-splitting-formula}, we obtain

$$V_{|E(F)|}(Q_{F//E'})=-V_{|E(F)|}(Q_{[F\setminus e_1,w,1]//E'})+V_{|E(F)|}(Q_{[F\setminus e_0,w,1]//E'}).$$

If we sum up these equations for all sets $E'\subset IE(F)\setminus\{e_0\}$, we get

$$\sum_{E'\subset IE(F)\setminus\{e_0\}}V_{|E(F)|}(Q_{F//E'})=$$
$$\sum_{E'\subset IE(F)\setminus\{e_0\}}-V_{|E(F)|}(Q_{[F\setminus e_1,w,1]//E'})+\sum_{E'\subset IE(F)\setminus\{e_0\}}V_{|E(F)|}(Q_{[F\setminus e_0,w,1]//E'})=$$
$$\sum_{E'\subset IE(F)\setminus\{e_0\}}-V_{|E(F)|}(Q_{[F\setminus e_1,w,1]//E'})+r([F\setminus e_0,w,1]).$$

For sets $E'$ that contain edge $e_0$, we have $F//E'=(F//e_0)//(E'\setminus \{e_0\})$, therefore

$$\sum_{E'\subset IE(F), e_0\in E'}V_{|E(F)|}(Q_{F//E'})=\sum_{E'\subset IE(F)\setminus\{e_0\}}V_{|E(F)|}(Q_{(F//e_0)//E'})=r(F//e_0).$$

Since the degree of $v$ is at least three, the edge $e_0$ is still an inner edge in $[F\setminus e_1,w,1]$. Analogously to the case of forest $F$, we have 
$$\sum_{E'\subset IE(F), e_0\in E'}V_{|E(F)|}(Q_{[F\setminus e_1,w,1]//E'})=r([F\setminus e_1,w,1]//e_0)=r(F//e_0).$$

Using the last three equations, we obtain:

$$r(F)=\sum_{E'\subset IE(F)\setminus\{e_0\}}V_{|E(F)|}(Q_{F//E'})+\sum_{E'\subset IE(F), e_0\in E'}V_{|E(F)|}(Q_{F//E'})=$$
$$=-\sum_{E'\subset IE(F)\setminus\{e_0\}}V_{|E(F)|}(Q_{[F\setminus e_1,w,1]//E'})+r([F\setminus e_1,w,1])+r(F//e_0)=
$$
$$-\sum_{E'\subset IE(F)\setminus\{e_0\}}V_{|E(F)|}(Q_{[F\setminus e_1,w,1]//E'})+r([F\setminus e_0,w,1])+r(F//e_0)+$$
$$+\left(-\sum_{E'\subset IE(F),e_0\in E'}V_{|E(F)|}(Q_{[F\setminus e_1,w,1]//E'})+r(F//e_0)\right)=$$
$$=r([F\setminus e_1,w,1])+r([F\setminus e_0,w,1])+2r(F//e_0).$$
\end{proof}

Now we can express the volume of the polytope $P_T$ using the function $r(F)$.
\begin{Theorem}\label{V-formula}
For any tree $T$, we have
$$V_{|E(T)|}(P_T)=\sum_{I'\subset I(T)}(-1)^{|I'|}\cdot r(G(I'))\cdot 2^{|E(G(I'))|-|I'|}.
$$   
\end{Theorem}
\begin{proof}

Let us recall that $$P_T=\fC_T\setminus\bigcup_{v\in I(T),A\in\mathcal{P}_{odd}(\mathcal{N}_T(v))} H^-_{v,A}. $$

 Let us denote $\mathcal M:=\{(v,A)\:\ v\in I(T), A\in\mathcal{P}_{odd}(\mathcal{N}_T(v))\} $. Also, for a set of inner vertices $I'\subset I(T)$ let us denote $$\mathcal A(I):=\{A\:\ I\rightarrow \mathcal P(E(T));\ \forall v\in I: A(v)\in \mathcal{P}_{odd}(\mathcal{N}_T(v))\}.$$
Moreover for $I'\subset I(T)$, $A\in \mathcal A(I)$, we denote $$\Gamma(A):=\{e=\{v,w\}\in IE(G(I'))\:\ e\in A(v)\triangle A(w)\},$$
where $\triangle$ stands for the symmetric difference of two sets.
 We will use the principle of inclusion and exclusion to compute the volume of $P_T$.

$$V_{|E(T)|}(P_T)=\sum_{M\subset \mathcal M} (-1)^{|M|} V_{|E(T)|}
\left( \fC_T \cap  \bigcap_{(v,A)\in M}H^-_{v,A}\right)= $$

$$=\sum_{I'\subset I(T)} (-1)^{|I'|}\sum_{A\in\mathcal A(I)} V_{|E(T)|}
\left( \fC_T \cap  \bigcap_{v\in I}H^-_{v,A(v)}\right)=$$

$$=\sum_{I'\subset I(T)} (-1)^{|I'|}\sum_{A\in\mathcal A(I)} V_{|E(G(I'))|}
\left(Q_{G(I')//\Gamma(A)}\right)=$$

$$=\sum_{I'\subset I(T)} (-1)^{|I'|}\sum_{E'\in IE(G(I'))}\sum_{\substack{A\in\mathcal A(I) \\ \Gamma(A)=E'}}V_{|E(G(I'))|}
\left(Q_{G(I')//E'}\right)=$$

$$=\sum_{I'\subset I(T)} (-1)^{|I'|}\sum_{E'\in IE(G(I'))} 2^{|E(G(I'))|-|I'|}\cdot V_{|E(G(I'))|}
\left(Q_{G(I')//E'}\right)=$$

$$=\sum_{I'\subset I(T)} (-1)^{|I'|} \cdot 2^{|E(G(I'))|-|I'|} \sum_{E'\in IE(G(I'))} V_{|E(G(I'))|}
\left(Q_{G(I')//E'}\right)=$$

$$=\sum_{I'\subset I(T)} (-1)^{|I'|} \cdot 2^{|E(G(I'))|-|I'|} \cdot r(G(I')).$$

In the first equality, we use Lemma~\ref{0volume} to sum only through such sets $M\in \mathcal M$ for which all vertices $v$ in pairs $(v,A)$ in $M$ are pairwise different.
Then we use Lemma~\ref{beautiful} to express the volume of $\fC_T \cap  \bigcap_{v\in I}H^-_{v,A(v)}$ and, in the end, we use the fact that there are exactly $2^{|E(G(I'))|-|I'|}$ functions $A\in \mathcal A(I)$ for which $\Gamma(A)=E'$.
\end{proof}
\begin{Remark}
Note that in the sum in Theorem \ref{V-formula} we always have a term for $I'=\emptyset$. Since from our definitions $V_0(Q_{\emptyset})=r(\emptyset)=1$, the corresponding term will always be equal to one. It corresponds to the volume of the cube $\fC_T$. A careful reader may verify that everything we have done so far works even in the case of the empty graph.
\end{Remark}

 \begin{Corollary}\label{Vstar}
For all $m,n\ge 1$ we have 
\begin{enumerate}[label=(\alph*)]
    \item $V_n(P_{S_n})=1-2^{n-1}/n!.$
    \item  $V_{m+n+1}(P_{S_{m,n}})=1-2^m/(m+1)!-2^n/(n+1)! +2^{m+n-1}\cdot \left(\frac{1}{(m+n+1)m!n!}+\frac{1}{(m+n+1)!}\right).$
\end{enumerate}   
 \end{Corollary}
 \begin{proof}
We use the formula from Theorem~\ref{V-formula}. In the case of $S_n$ there is only one inner vertex, thus, the sum has only two terms:

$$V_n(P_{S_n})=r(\emptyset)\cdot 2^0-r(S_n)\cdot 2^{n-1}=1-\frac{2^{n-1}}{n!}.$$

In the case of $S_{m,n}$ the sum has 4 terms:
$$V_n(P_{S_{m,n}})=r(\emptyset)\cdot 2^0-r(S_{m+1})\cdot 2^{m}-r(S_{n+1})\cdot 2^{n}+r(S_{m,n})\cdot 2^{m+n-1}=$$
$$=1-\frac{2^m}{(m+1)!}-\frac{2^n}{(n+1)!} +2^{m+n-1}\cdot \left(\frac{1}{(m+n+1)m!n!}+\frac{1}{(m+n+1)!}\right).$$
For the value of $r$ on stars and double stars, we used Lemma~\ref{qstar}.
 \end{proof}
 \begin{Example}\label{example-V}
We illustrate the computation of $V_7(P_T)$, as a continuation of Examples \ref{example-P} and \ref{r(T)}. We refer now to Figure~\ref{272}, where we highlight the inner vertices.
\vspace{.1in}
\begin{figure}[h]
\centering
\begin{tikzpicture}
\foreach \x in {2,4,6}{
\filldraw [red] (\x,1) circle (3pt);}
\draw [black, thick] (2,1)--(1,2);
\draw [black, thick] (1,0)--(2,1);
\draw [black, thick] (2,1)--(4,1);
\node [right, black, ultra thick] at (1.9,1.3) {$v_1$};
\node [right, black, ultra thick] at (4.1,1.3) {$v_2$};
\node [right, black, ultra thick] at (5.5,1.3) {$v_3$};
\draw [black, thick] (4,1)--(4,2);
\draw [black, thick] (4,1)--(6,1);
\draw [black, thick] (6,1)--(7,0);
\draw [black, thick] (7,2)--(6,1);
\end{tikzpicture}
\caption{}
\label{272}
\end{figure}

The set $I(T)=\{v_1,v_2,v_3\}$ has 8 subsets. Thus:

$$V_7(P_T)=-r(T)\cdot 2^{7-3}+r(S_{2,2})\cdot 2^{5-2}+r(S_{3}\cup S_3)\cdot 2^{6-2}+r(S_{2,2})\cdot 2^{5-2}-$$
$$-r(S_{3})\cdot 2^{3-1}-r(S_{3})\cdot 2^{3-1}-r(S_{3})\cdot 2^{3-1}+1=$$
$$=-\frac{16\cdot 102}{7!}+2\cdot8\cdot\left(\frac{1}{2!\cdot2!\cdot5}+\frac{1}{5!}\right)+16\left(\frac{1}{3!}\right)^2-3\cdot4\cdot\frac{1}{3!}+1=\frac{272}{7!}.$$

We conclude that the lattice volume of $P_T$ is 272 and its lattice volume in the lattice $L_T$ is $272/8=34$. Thus, the phylogenetic degree of the algebraic variety $X_T$ is 34.

 \end{Example}
 
 We will mention one important, but trivial fact:
\begin{Lemma}\label{deg2}
Let $T$ be a tree that contains a vertex of degree two. Then, $V_{|E(T)|}(P_T)=0$. 
\end{Lemma}

\begin{proof}
Let $v$ be a vertex of degree two and $e_1$ and $e_2$ be edges adjacent to $v$. From the inequalities $0\le x_{e_1}-x_{e_2}\le 0$, it follows that $x_{e_1}=x_{e_2}$. This implies that $P_T$ is not full-dimensional and, therefore, its volume is zero.
\end{proof}

We define polytopes $P_T$ only for the trees $T$ because that is what is meaningful for the binary Jukes-Cantor model. However, we can similarly define polytope $P_F$ for any forest $F$ using the same facet description. It is easy to verify that we never used the fact that $T$ is a tree, so Theorem \ref{V-formula} holds also for a forest $F$. Moreover, if $F$ is a forest with connected components $T_1,\dots,T_k$, from the definition we have $P_F=\prod_{i=1}^k P_{T_i}$. Thus, it may seem that considering polytopes associated to forests $P_F$ is pointless, but it will allow us to formulate the result from the next section more easily.

\section{Recursive formula for the volume}\label{anotherformula}
\vspace{.2in}

 For a forest $F$ and a vertex $v\in I(F)$ we will denote by $F\ominus v$ the forest obtained by removing the vertex $v$ but keeping the edges adjacent to $v$ as leaves.
 Thus, $E(F\ominus v)=E(F)$, but $I(F\ominus v)=I(F)\setminus\{v\}$.

\begin{Example}
Consider a tree $T$ with three inner vertices of degrees 3,4, and 3. Let $v$ be the middle inner vertex. Then the graph $T\ominus v$ is a forest with 4 components: two 3-stars and two 1-stars, as in \ref{exmaple-ominus}.
\begin{figure}[H]
\begin{minipage}[c]{0.5\linewidth}
\centering
\begin{tikzpicture}
\foreach \x in {4}{
\filldraw [red] (\x,1) circle (3pt);}
\draw [black, thick] (1,0)--(2,1)--(1,2);
\draw [black, thick] (2,1)--(4,1)--(3,2)--(4,1)--(5,2);
\draw [black, thick] (4,1)--(6,1)--(7,2)--(6,1)--(7,0);
\end{tikzpicture}
\caption*{$T$ with marked vertex $v$}
\end{minipage}\hfill
\begin{minipage}[c]{0.5\linewidth}
\centering
\begin{tikzpicture}
\draw [black, thick] (1,0)--(2,1)--(1,2);
\draw [black, thick] (2,1)--(3.4,1);
\draw [black, thick] (3.8,1)--(3.8,2);
\draw [black, thick] (4.3,1)--(4.3,2);
\draw [black, thick] (4.6,1)--(6,1)--(7,2)--(6,1)--(7,0);
\end{tikzpicture}
\caption*{$T\ominus v$}
\end{minipage}
\caption{}
\label{exmaple-ominus}
\end{figure}
\end{Example}

 \begin{Theorem}\label{v-splitting}
Let $F$ be a forest with an inner edge $e_0$ connecting the vertices $v$ and $w$. Let $e_1$ be an edge adjacent to $v$ such that its other end is a leaf. Suppose that the degree of $v$ is at least three. Then the following equation holds (see also Figure \ref{picture-Vsplit}):

$$
V_{|E(F)|}(P_F)=-V_{|E(F)|}(P_{[F\setminus e_1,w,1]})+V_{|E(F)|}(P_{[F\setminus e_0,w,1]})-$$
$$-V_{|E(F)|}(P_{F//e_0})+V_{|E(F)|}(P_{F\ominus w})+V_{|E(F)|}(P_{[F\setminus e_1,w,1]\ominus v}).$$
 \end{Theorem}
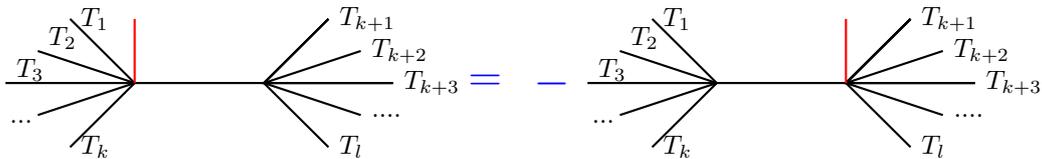
 \begin{figure}[H]
\begin{minipage}[c]{0.5\linewidth}
\centering
\begin{tikzpicture}[scale=0.85]
\draw [black, thick] (1,0)--(2,1)--(1,2);
\draw [black, thick] (2,1)--(4,1);
\draw [black, thick] (2,1)--(0.5,1.5);
\draw [black, thick] (2,1)--(0.5,0.5);
\draw [black, thick] (0,1)--(2,1);
\draw [red, thick] (2,1)--(2,2);
\draw [black, thick] (4,1)--(5,2)--(4,1)--(6,1)--(4,1)--(5,0);
\draw [black, thick] (4,1)--(5.5,1.5);
\draw [black, thick] (4,1)--(5.5,0.5);
\node [right, black, ultra thick] at (0.5,1.7) {$T_2$};
\node [right, black, ultra thick] at (1,2) {$T_1$};
\node [right, black, ultra thick] at (0,1.2) {$T_3$};
\node [right, black, ultra thick] at (-0.1,0.4) {$...$};
\node [right, black, ultra thick] at (1,0) {$T_k$};

\node [right, black, ultra thick] at (5.5,0.5) {$....$};
\node [right, black, ultra thick] at (5.5,1.5) {$T_{k+2}$};
\node [right, black, ultra thick] at (5,2) {$T_{k+1}$};
\node [right, black, ultra thick] at (6,1) {$T_{k+3}$};
\node [right, black, ultra thick] at (5, 0) {$T_{l}$};
\node [right, blue, ultra thick] at (7,1) {\huge $=$};

\end{tikzpicture}
\end{minipage}\hfill
\begin{minipage}[c]{0.5\linewidth}
\centering
\begin{tikzpicture}[scale=0.85]

\draw [black, thick] (1,0)--(2,1)--(1,2);
\draw [black, thick] (2,1)--(4,1);
\draw [black, thick] (2,1)--(0.5,1.5);
\draw [black, thick] (2,1)--(0.5,0.5);
\draw [black, thick] (0,1)--(2,1);
\draw [black, thick] (4,1)--(5,2)--(4,1)--(6,1)--(4,1)--(5,0);
\draw [black, thick] (4,1)--(5.5,1.5);
\draw [black, thick] (4,1)--(5.5,0.5);
\node [right, black, ultra thick] at (0.5,1.7) {$T_2$};
\node [right, black, ultra thick] at (1,2) {$T_1$};
\node [right, black, ultra thick] at (0,1.2) {$T_3$};
\node [right, black, ultra thick] at (-0.1,0.4) {$...$};
\node [right, black, ultra thick] at (1,0) {$T_k$};
\draw [red, thick] (4,1)--(4,2);
\node [right, black, ultra thick] at (5.5,0.5) {$....$};
\node [right, black, ultra thick] at (5.5,1.5) {$T_{k+2}$};
\node [right, black, ultra thick] at (5,2) {$T_{k+1}$};
\node [right, black, ultra thick] at (6,1) {$T_{k+3}$};
\node [right, black, ultra thick] at (5, 0) {$T_{l}$};
\node [right, blue, ultra thick] at (-1,1) {\huge $-$};
\end{tikzpicture}
\end{minipage}
\end{figure}
\begin{figure}[H]
\begin{minipage}[c]{0.5\linewidth}
\centering
\begin{tikzpicture}[scale=0.85]
\draw [black, thick] (1,0)--(2,1)--(1,2);
\draw [black, thick] (2,1)--(0.5,1.5);
\draw [black, thick] (2,1)--(0.5,0.5);
\draw [black, thick] (0,1)--(2,1);
\draw [red, thick] (2,1)--(2,2);
\draw [black, thick] (4,1)--(5,2)--(4,1)--(6,1)--(4,1)--(5,0);
\draw [black, thick] (4,1)--(5.5,1.5);
\draw [black, thick] (4,1)--(5.5,0.5);
\node [right, black, ultra thick] at (0.5,1.7) {$T_2$};
\node [right, black, ultra thick] at (1,2) {$T_1$};
\node [right, black, ultra thick] at (0,1.2) {$T_3$};
\node [right, black, ultra thick] at (-0.1,0.4) {$...$};
\node [right, black, ultra thick] at (1,0) {$T_k$};
\draw [red, thick] (4,1)--(4,2);
\node [right, black, ultra thick] at (5.5,0.5) {$....$};
\node [right, black, ultra thick] at (5.5,1.5) {$T_{k+2}$};
\node [right, black, ultra thick] at (5,2) {$T_{k+1}$};
\node [right, black, ultra thick] at (6,1) {$T_{k+3}$};
\node [right, black, ultra thick] at (5, 0) {$T_{l}$};
\node [right, blue, ultra thick] at (-1,1.1) {\huge $+$};
\end{tikzpicture}
\end{minipage}\hfill
\begin{minipage}[c]{0.5\linewidth}
\centering
\begin{tikzpicture}
\draw [black, thick] (1,0)--(2,1)--(1,2);
\draw [black, thick] (2,1)--(0.5,1.5);
\draw [black, thick] (2,1)--(0.5,0.5);
\draw [black, thick] (0,1)--(2,1);
\draw [black, thick] (2,1)--(3,2)--(2,1)--(4,1)--(2,1)--(3,0);
\draw [black, thick] (2,1)--(3.5,1.5);
\draw [black, thick] (2,1)--(3.5,0.5);
\node [right, black, ultra thick] at (0.5,1.7) {$T_2$};
\node [right, black, ultra thick] at (1,2) {$T_1$};
\node [right, black, ultra thick] at (0,1.2) {$T_3$};
\node [right, black, ultra thick] at (-0.1,0.4) {$...$};
\node [right, black, ultra thick] at (1,0) {$T_k$};
\draw [red, thick] (2,1)--(2.4,2);
\draw [red, thick] (2,1)--(1.6,2);
\node [right, black, ultra thick] at (3.5,0.5) {$....$};
\node [right, black, ultra thick] at (3.5,1.5) {$T_{k+2}$};
\node [right, black, ultra thick] at (3,2) {$T_{k+1}$};
\node [right, black, ultra thick] at (4,1) {$T_{k+3}$};
\node [right, black, ultra thick] at (3, 0) {$T_{l}$};
\node [right, blue, ultra thick] at (-1,1.1) {\huge $-$};
\end{tikzpicture}
\end{minipage}

\end{figure}

\begin{figure}[H]
\begin{minipage}[c]{0.5\linewidth}
\centering
\begin{tikzpicture}[scale=0.85]
\draw [black, thick] (1,0)--(2,1)--(1,2);
\draw [black, thick] (2,1)--(3,1);
\draw [black, thick] (2,1)--(0.5,1.5);
\draw [black, thick] (2,1)--(0.5,0.5);
\draw [black, thick] (0,1)--(2,1);
\draw [red, thick] (2,1)--(2,2);
\node [right, black, ultra thick] at (0.5,1.7) {$T_2$};
\node [right, black, ultra thick] at (1,2) {$T_1$};
\node [right, black, ultra thick] at (0,1.2) {$T_3$};
\node [right, black, ultra thick] at (-0.1,0.4) {$...$};
\node [right, black, ultra thick] at (1,0) {$T_k$};
\draw [black, thick] (4,0.5)--(5,0.5);
\draw [black, thick] (4,1)--(5,1);
\draw [black, thick] (4,1.5)--(5,1.5);
\draw [black, thick] (4,2)--(5,2);
\node [right, black, ultra thick] at (5,2) {$T_{k+1}$};
\node [right, black, ultra thick] at (5,1.5) {$T_{k+2}$};
\node [right, black, ultra thick] at (5,1) {$...$};
\node [right, black, ultra thick] at (5,0.5) {$T_{l}$};
\node [right, blue, ultra thick] at (-1,1.2) {\huge $+$};
\end{tikzpicture}
\end{minipage}\hfill
\begin{minipage}[c]{0.5\linewidth}
\centering
\begin{tikzpicture}[scale=0.85]
\draw [black, thick] (1,0.5)--(2,0.5);
\draw [black, thick] (1,1)--(2,1);
\draw [black, thick] (1,1.5)--(2,1.5);
\draw [black, thick] (1,2)--(2,2);
\node [right, black, ultra thick] at (0.3,2) {$T_{1}$};
\node [right, black, ultra thick] at (0.3,1.5) {$T_{2}$};
\node [right, black, ultra thick] at (0.3,1) {$...$};
\node [right, black, ultra thick] at (0.3,0.5) {$T_{k}$};
\draw [black, thick] (3,1)--(4,1);
\draw [black, thick] (4,1)--(5,2)--(4,1)--(6,1)--(4,1)--(5,0);
\draw [black, thick] (4,1)--(5.5,1.5);
\draw [black, thick] (4,1)--(5.5,0.5);
\draw [red, thick] (4,1)--(4,2);
\node [right, black, ultra thick] at (5.5,0.5) {$....$};
\node [right, black, ultra thick] at (5.5,1.5) {$T_{k+2}$};
\node [right, black, ultra thick] at (5,2) {$T_{k+1}$};
\node [right, black, ultra thick] at (6,1) {$T_{k+3}$};
\node [right, black, ultra thick] at (5, 0) {$T_{l}$};
\node [right, blue, ultra thick] at (-1,1.2) {$-$};
\node [right, blue, ultra thick] at (-1,1.2) {\huge $+$};
\end{tikzpicture}
\end{minipage}
\caption{Recursive formula for $V(P_F)$}
\label{picture-Vsplit}
\end{figure}
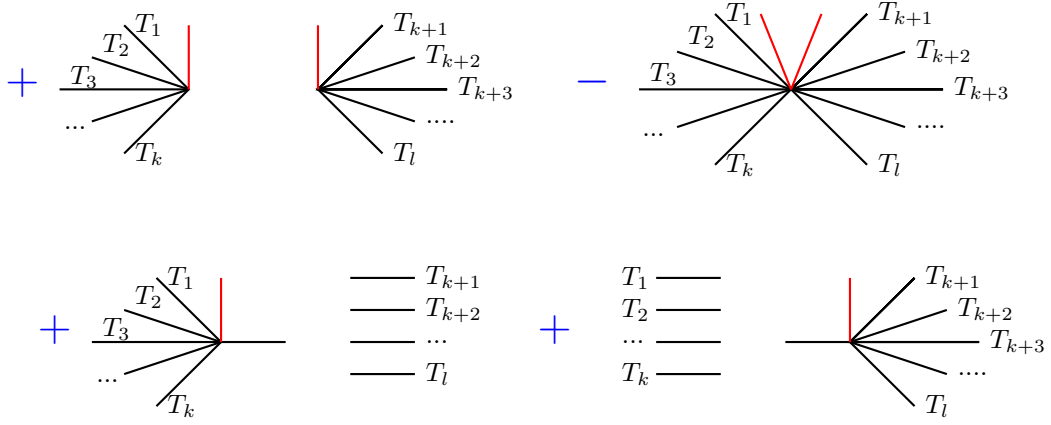

 \begin{proof}
We will use the formula from Theorem \ref{V-formula}. We will split the terms in the sum depending on which of the vertices $v,w$ are contained in the set $I'$. We will be using graphs generated by a set of inner vertices $I'$ in different forests. Thus, we will extend our notation to $G_F(I')$ which will denote the graph generated by inner vertices $I'$ in the forest $F$. We have the following:

$$V_{|E(F)|}(P_F)=\sum_{I'\subset I(F)\setminus\{v,w\}}(-1)^{|I'|}\cdot r(G_F(I'))\cdot 2^{|E(G_F(I'))|-|I'|}+$$
$$+\sum_{I'\subset I(F)\setminus\{v\}, w\in I'}(-1)^{|I'|}\cdot r(G_F(I'))\cdot 2^{|E(G_F(I'))|-|I'|}+$$
$$+\sum_{I'\subset I(F)\setminus\{w\}, v\in I'}(-1)^{|I'|}\cdot r(G_F(I'))\cdot 2^{|E(G_F(I'))|-|I'|}+$$
$$+\sum_{I'\subset I(F),v,w\in I'}(-1)^{|I'|}\cdot r(G_F(I'))\cdot 2^{|E(G_F(I'))|-|I'|},$$

$$V_{|E(F)|}(P_{[F\setminus e_1,w,1]})=\sum_{I'\subset I(F)\setminus\{v,w\}}(-1)^{|I'|}\cdot r(G_{[F\setminus e_1,w,1]}(I'))\cdot 2^{|E(G_{[F\setminus e_1,w,1]}(I'))|-|I'|}+$$
$$+\sum_{I'\subset I(F)\setminus\{v\}, w\in I'}(-1)^{|I'|}\cdot r(G_{[F\setminus e_1,w,1]}(I'))\cdot 2^{|E(G_{[F\setminus e_1,w,1]}(I'))|-|I'|}+$$
$$+\sum_{I'\subset I(F)\setminus\{w\}, v\in I'}(-1)^{|I'|}\cdot r(G_{[F\setminus e_1,w,1]}(I'))\cdot 2^{|E(G_{[F\setminus e_1,w,1]}(I'))|-|I'|}+$$
$$+\sum_{I'\subset I(F),v,w\in I'}(-1)^{|I'|}\cdot r(G_{[F\setminus e_1,w,1]}(I'))\cdot 2^{|E(G_{[F\setminus e_1,w,1]}(I'))|-|I'|},$$

$$V_{|E(F)|}(P_{[F\setminus e_0,w,1]})=\sum_{I'\subset I(F)\setminus\{v,w\}}(-1)^{|I'|}\cdot r(G_{[F\setminus e_0,w,1]}(I'))\cdot 2^{|E(G_{[F\setminus e_0,w,1]}(I'))|-|I'|}+$$
$$+\sum_{I'\subset I(F)\setminus\{v\}, w\in I'}(-1)^{|I'|}\cdot r(G{[F\setminus e_0,w,1]}F(I'))\cdot 2^{|E(G_{[F\setminus e_0,w,1]}(I'))|-|I'|}+$$
$$+\sum_{I'\subset I(F)\setminus\{w\}, v\in I'}(-1)^{|I'|}\cdot r(G_{[F\setminus e_0,w,1]}(I'))\cdot 2^{|E(G_{[F\setminus e_0,w,1]}(I'))|-|I'|}+$$
$$+\sum_{I'\subset I(F),v,w\in I'}(-1)^{|I'|}\cdot r(G_{[F\setminus e_0,w,1]}(I'))\cdot 2^{|E(G_{[F\setminus e_0,w,1]}(I'))|-|I'|}.$$

In the following formula, we will denote by $vw$ the contracted vertex in $F//e_0$:
$$V_{|E(F)|}(P_{F//e_0})=\sum_{I'\subset I({F//e_0})\setminus\{vw\}}(-1)^{|I'|}\cdot r(G_{F//e_0}(I'))\cdot 2^{|E(G_{F//e_0}(I'))|-|I'|}+$$
$$+\sum_{I'\subset I({F//e_0}), vw\in I'}(-1)^{|I'|}\cdot r(G_{F//e_0}(I'))\cdot 2^{|E(G_{F//e_0}(I'))|-|I'|},$$

$$V_{|E(F)|}(P_{F\ominus w})=\sum_{I'\subset I(F)\setminus\{v\}}(-1)^{|I'|}\cdot r(G_{F\ominus w}(I'))\cdot 2^{|E(G_{F\ominus w}(I'))|-|I'|}+$$
$$+\sum_{I'\subset I(F), v\in I'}(-1)^{|I'|}\cdot r(G_{F\ominus w}(I'))\cdot 2^{|E(G_{F\ominus w}(I'))|-|I'|},$$

$$V_{|E(F)|}(P_{[F\setminus e_1,w,1]\ominus v})=\sum_{I'\subset I(F)\setminus\{w\}}(-1)^{|I'|}\cdot r(G_{[F\setminus e_1,w,1]\ominus v}(I'))\cdot 2^{|E(G_{[F\setminus e_1,w,1]\ominus v}(I'))|-|I'|}+$$
$$+\sum_{I'\subset I(F), w\in I'}(-1)^{|I'|}\cdot r(G_{[F\setminus e_1,w,1]\ominus v}(I'))\cdot 2^{|E(G_{[F\setminus e_1,w,1]\ominus v}(I'))|-|I'|}.$$

Now we will compare these 18 terms in a few equalities:

\textbf{Terms for $I'$ that contain both $v$ and $w$}:
Consider any $I'$ such that $v,w\in I'$.
Note that $|E(G_{[F\setminus e_1,w,1]}(I'))|=|E(G_{[F\setminus e_0,w,1]}(I'))|=|E(G_{F//e_0}(I'\setminus \{v,w\}\cup\{vw\}))|$. Moreover,
$$G_{[F\setminus e_1,w,1]}(I')=[G_F(I')\setminus e_1,w,1],$$ $$G_{[F\setminus e_0,w,1]}(I')=[G_F(I')\setminus e_0,w,1],$$
$$G_{F//e_0}(I'\setminus \{v,w\}\cup\{vw\})=G_F(I')//e_0.$$

By Lemma~\ref{Q-splitting-formula}, we have

$$r(G_F(I'))=-r([G_F(I')\setminus e_1,w,1])+r([G_F(I')\setminus e_0,w,1])+2r(G_F(I')//e_0).$$

By using three equalities above and multiplying by $(-1)^{|I'|}\cdot 2^{|E(G_F(I'))|-|I'|}$, we obtain:

$$(-1)^{|I'|}\cdot r(G_F(I'))\cdot 2^{|E(G_F(I'))|-|I'|}=-(-1)^{|I'|}\cdot r(G_{[F\setminus e_1,w,1]}(I'))\cdot 2^{|E(G_{[F\setminus e_1,w,1]}(I'))|-|I'|}
$$
$$+(-1)^{|I'|}\cdot r(G_{[F\setminus e_0,w,1]}(I'))\cdot 2^{|E(G_{[F\setminus e_0,w,1]}(I'))|-|I'|}-(-1)^{|I'|}\cdot r(G_{F//e_0}(I'))\cdot 2^{|E(G_{F//e_0}(I'))|-|I'|}.$$

\textbf{Terms for $I'$ that contain $v$ but not $w$}: Consider any such set $I'$. Note that $G_F(I')=G_{F\ominus w}(I')$ and $G_{[F\setminus e_1,w,1]}(I')=G_{[F\setminus e_0,w,1]}(I')$, thus

$$(-1)^{|I'|}\cdot r(G_F(I'))\cdot 2^{|E(G_F(I'))|-|I'|}=(-1)^{|I'|}\cdot r(G_{F\ominus w}(I'))\cdot 2^{|E(G_{F\ominus w}(I'))|-|I'|},
$$

$$0=-(-1)^{|I'|}\cdot r(G_{[F\setminus e_1,w,1]}(I'))\cdot 2^{|E(G_{[F\setminus e_1,w,1]}(I'))|-|I'|}+$$
$$+(-1)^{|I'|}\cdot r(G_{[F\setminus e_0,w,1]}(I'))\cdot 2^{|E(G_{[F\setminus e_0,w,1]}(I'))|-|I'|}.$$

\textbf{Terms for $I'$ that contain $w$ but not $v$}:  Consider any such set $I'$. Note that $G_F(I')=G_{[F\setminus e_0,w,1]}(I')$ and $G_{[F\setminus e_1,w,1]}(I')=G_{[F\setminus e_1,w,1]\ominus v}(I')$, thus

$$(-1)^{|I'|}\cdot r(G_F(I'))\cdot 2^{|E(G_F(I'))|-|I'|}=(-1)^{|I'|}\cdot r(G_{[F\setminus e_0,w,1]}(I'))\cdot 2^{|E(G_{[F\setminus e_0,w,1]}(I'))|-|I'|},
$$

$$0=-(-1)^{|I'|}\cdot r(G_{[F\setminus e_1,w,1]}(I'))\cdot 2^{|E(G_{[F\setminus e_1,w,1]}(I'))|-|I'|}+$$
$$+(-1)^{|I'|}\cdot r(G_{[F\setminus e_1,w,1]\ominus v}(I'))\cdot 2^{|E(G_{[F\setminus e_1,w,1]\ominus v}(I'))|-|I'|}.$$

\textbf{Terms for $I'$ that do not contain $v$ and $w$}: Consider any such set $I'$. Note that 
$$G_F(I')=G_{[F\setminus e_0,w,1]}(I')=G_{[F\setminus e_1,w,1]}(I')=G_{F//e_0}(I')=G_{F\ominus w}(I')=G_{[F\setminus e_1,w,1]\ominus v}(I').$$

Therefore, we obtain
$$(-1)^{|I'|}\cdot r(G_F(I'))\cdot 2^{|E(G_F(I'))|-|I'|}=-(-1)^{|I'|}\cdot r(G_{[F\setminus e_1,w,1]}(I'))\cdot 2^{|E(G_{[F\setminus e_1,w,1]}(I'))|-|I'|}
$$
$$+(-1)^{|I'|}\cdot r(G_{[F\setminus e_0,w,1]}(I'))\cdot 2^{|E(G_{[F\setminus e_0,w,1]}(I'))|-|I'|}-(-1)^{|I'|}\cdot r(G_{F//e_0}(I'))\cdot 2^{|E(G_{F//e_0}(I'))|-|I'|}+$$
$$+(-1)^{|I'|}\cdot r(G_{F\ominus w}(I'))\cdot 2^{|E(G_{F\ominus w}(I'))|-|I'|}+(-1)^{|I'|}\cdot r(G_{[F\setminus e_1,w,1]\ominus v}(I'))\cdot 2^{|E(G_{[F\setminus e_1,w,1]\ominus v}(I'))|-|I'|}.$$

By summing up the above equalities for all $I'$, we get the formula from the theorem statement.

 \end{proof}

\begin{Remark}   
 The formula given in Theorem~\ref{v-splitting} allows us to recursively compute the volumes of polytopes $P_T$ without the need to compute the numbers $r(F)$ or volumes of polytopes $Q_F$. Note that three of five graphs on the right-hand side of the formula are not connected, thus computing the volume is immediately reduced to computing the volume of smaller graphs. The graph $F//e_0$ has one less inner vertex than $F$, so is also simpler. The only problem is the graph $[F\setminus e_1,w,1]$. However, if we take any inner vertex $v$ with a single adjacent inner edge, we can use the formula to move all of its leaves to the next vertex and eventually get to the situation where the degree of $v$ is two.
 \end{Remark}

Using Theorem~\ref{v-splitting} for specific graphs one can obtain a simpler relation between volumes of polytopes $P_T$. However, there seems to be no easy way to obtain such relations, which would make the volume computations much faster.
We provide a few such relations in the following results:

\begin{Corollary}\label{easy-Vformula}
Let $F$ be a forest with inner vertices $v,w$ such that $v$ is of degree 3 and $w$ is of degree 2. Let $e_0$ be the edge connecting $v$ and $w$. Moreover, assume that one of the other edges adjacent to $v$ is a leaf. Denote this edge by $e_1$.

Then $$V_{|E(F)|}(P_{F//e_0})=V_{|E(F)|}(P_{F\ominus w})+V_{|E(F)|}(P_{[F\setminus e_1,w,1]\ominus v}).$$
\end{Corollary}
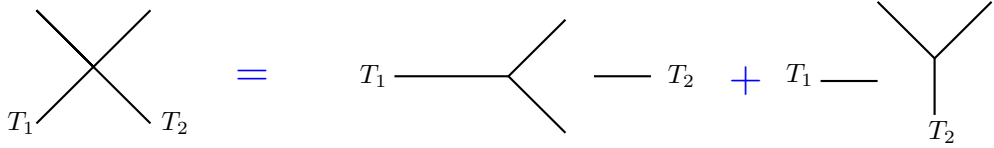
\begin{figure}[H]
\begin{minipage}[c]{0.35\linewidth}
\centering
\begin{tikzpicture}[scale=0.75]
\draw [black, thick] (1,0)--(2,1)--(1,2)--(2,1)--(3,2);
\draw [black, thick] (2,1)--(3,0);
\node [right, black, ultra thick] at (0.3, 0) {$T_1$};
\node [right, black, ultra thick] at (3, 0) {$T_2$};
\node [right, blue, ultra thick] at (4.3,0.8) {\huge $=$};
\end{tikzpicture}
\end{minipage}\hfill
\begin{minipage}[c]{0.3\linewidth}
 \centering
\begin{tikzpicture}[scale=0.75]
\draw [black, thick] (0,1)--(2,1)--(3,2);
\draw [black, thick] (2,1)--(3,0);
\draw [black, thick] (3.5,1)--(4.5,1);
\node [right, black, ultra thick] at (-0.8, 1) {$T_1$};
\node [right, black, ultra thick] at (4.6, 1) {$T_2$};
\end{tikzpicture}  
\end{minipage}\hfill
\begin{minipage}[c]{0.3\linewidth}
 \centering
\begin{tikzpicture}[scale=0.75]
\draw [black, thick] (0,0.6)--(1,0.6);
\draw [black, thick] (1,2)--(2,1)--(3,2);
\draw [black, thick] (2,1)--(2,0);
\node [right, black, ultra thick] at (-0.8, 0.7) {$T_1$};
\node [right, black, ultra thick] at (1.7, -0.3) {$T_2$};
\node [right, blue, ultra thick] at (-1.8,0.6) {\huge $+$};
\end{tikzpicture}  
\end{minipage}
\end{figure}
\begin{proof}
It is sufficient to observe that the graphs $F$, $[F\setminus e_1,w,1]$ and $[F\setminus e_0,w,1]$ all contain a vertex of degree two and, therefore, the volume of the corresponding polytopes is 0. Then the statement is a direct consequence of Theorem \ref{v-splitting}.
\end{proof}

\begin{Corollary}\label{easy_formula_2}
Let $T$ be a tree with a leaf $w$. Then $$\frac13 V_{|E(T)|}(P_T)+V_{|E(T)|+2}(P_{[T,w,2]})=V_{|E(T)|+3}(P_{[T,w,3]}).$$
\end{Corollary}
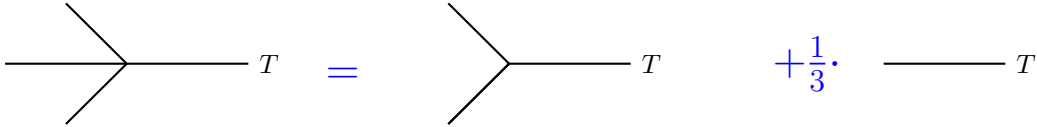
\begin{figure}[H]
\begin{minipage}[c]{0.35\linewidth}
\centering
\begin{tikzpicture}[scale=0.8]
\draw [black, thick] (1,0)--(2,1)--(1,2);
\draw [black, thick] (2,1)--(4,1);
\draw [black, thick] (0,1)--(2,1);
\node [right, black, ultra thick] at (4, 1) {$T$};
\node [right, blue, ultra thick] at (5.1,0.8) {\huge $=$};
\end{tikzpicture}
\end{minipage}\hfill
\begin{minipage}[c]{0.3\linewidth}
 \centering
 \begin{tikzpicture}[scale=0.8]
\draw [black, thick] (1,0)--(2,1)--(1,2);
\draw [black, thick] (2,1)--(4,1);
\draw [black, thick] (1,0)--(2,1);
\node [right, black, ultra thick] at (4, 1) {$T$};
\end{tikzpicture}  
\end{minipage}\hfill
\begin{minipage}[c]{0.3\linewidth}
 \centering
 \begin{tikzpicture}[scale=0.8]
\draw [black, thick] (2,1)--(4,1);
\node [right, black, ultra thick] at (4, 1) {$T$};
\node [right, blue, ultra thick] at (0,1) {\huge $+\frac{1}{3}\cdot $};
\end{tikzpicture}  
\end{minipage}
\end{figure}

\begin{proof}
Let $v$ be the new leaf of a tree $[T,w,1]$. Consider the graph $T':=[[T,w,1],v,2]$. Let us denote by $e_0$ the edge connecting $v$ and $w$, and by $e_1$ one of the other two edges adjacent to $v$. We will apply Theorem \ref{easy-Vformula} for the tree $T'$ and edges $e_0, e_1$. 

We have $T'//e_0\cong [T,w,3]$. Furthermore, the forest $T'\ominus w$ consists of two components, the tree $T$ and a 3-star. The forest $[T'\setminus e_1,w,1]\ominus v$ consists of two components, one is isomorphic to $[T,w,2]$ and the other is 1-star. Thus,

 $$V_{|E(T)|+3}(P_{T'//e_0})= V_{|E(T)|+3}(P_{[T,w,3]}),$$
 $$V_{|E(T)|+3}(P_{T'\ominus w})= V_{|E(T)|}(P_T)\cdot V_3(S_3)=\frac 13 V_{|E(T)|}(P_T),$$
  $$V_{|E(T)|+3}(P_{[T'\setminus e_1,w,1]\ominus v})= V_{|E(T)|+2}(P_{[T,w,2]})\cdot V_1(S_1)=V_{|E(T)|+2}(P_{[T,w,2]}).$$

By Corollary \ref{easy-Vformula}, we have:

$$V_{|E(T)|+3}(P_{T'//e_0})=V_{|E(T)|+3}(P_{T'\ominus w})V_{|E(T)|+3}(P_{[T'\setminus e_1,w,1]\ominus v})\Leftrightarrow$$
$$\Leftrightarrow V_{|E(T)|+3}(P_{[T,w,3]})=V_{|E(T)|+2}(P_{[T,w,2]})+\frac 13 V_{|E(T)|}(P_T),$$

which proves the statement.
\end{proof}

\subsection{Examples.} We will end this section with two tables of phylogenetic degrees for all trees with at most ten edges. The formulas for star and double-star trees was already given in Corollary \ref{Vstar}.

We generalize the notion of a double star to the path-tree $S_{a_1,\dots,a_k}$. It is a tree with $k$ inner vertices $v_1,\dots, v_k$ which form a path. Moreover, each vertex $v_i$ has exactly $a_k$ leaves. It is easy to check that the cases $k=1$ and $k=2$ correspond to star and double-star. For instance, $S_{2,1,2}$ is the graph from Example~\ref{example-P}.

We list only trees where each inner vertex has a degree at least three. All of the path-trees are listed in Table~\ref{table:1}.

\begin{table}[h!]
\centering
\begin{tabular}{||c c c c||} 
 \hline
 Tree $T$ & Volume of $P_T$ & Lattice volume in $\ZZ^ {|E(T)|}$ & Phylogenetic degree of $X_T$ \\ [0.5ex] 
 \hline\hline
  $S_{2,1,2}$ & 17/315 & 272 & 34 \\ 
  $S_{3,1,2}$ & 31/315 & 3968 & 496 \\
  $S_{2,2,2}$ & 4/45 & 3584 & 448 \\
  $S_{4,1,2}$ & 49/405 & 43904 & 5488\\  
  $S_{3,2,2}$ & 23/135 & 61824 & 7728 \\
  $S_{3,1,3}$ & 34/189 & 65280 & 8160\\ 
  $S_{2,3,2}$ & 59/567 & 37760 & 4720 \\ 
  $S_{2,1, 1,2}$ & 62/2835 & 7936 & 496\\ 
  $S_{5,1,2}$ & 1838/14175 & 470528 & 58816 \\
  $S_{4,2,2}$ & 1021/4725 & 784128 & 98016\\
  $S_{4,1,3}$ & 628/2835 & 803840 & 100480 \\
  $S_{3,2,3}$ & 44/135 & 1182720 & 147840 \\ 
   $S_{3,3,2}$ & 116/567 & 742400 & 92800 \\
  $S_{2,4,2}$ & 5129/50400 & 396288 & 49536\\
  $S_{3,1, 1,2}$ & 113/2835 & 144640 & 9040\\
   $S_{2,2, 1,2}$ & 169/4725 & 129792 & 8112 \\ 
 \hline
\end{tabular}
\caption{Path-trees with at most 10 edges.}
\label{table:1}
\end{table}

There are also three trees with at most 10 edges which are not path-trees and they are listed in Table~\ref{table:2}.

\begin{table}[H]
\centering
\begin{tabular}{||c c c c||} 
 \hline
 Tree $T$ & Volume of $P_T$ & Lattice volume in $\ZZ^ {|E(T)|}$ & Phylogenetic degree of $X_T$\\ [0.3ex] 
 \hline\hline
  \begin{tikzpicture}[scale=0.6]
\draw [black, thick] (1,0)--(2,1)--(1,2);
\draw [black, thick] (2,1)--(3,1)--(4,2)--(3,3)--(4,2)--(5,3);
\draw [black, thick] (3,1)--(4,0)--(3,-1)--(4,0)--(5,-1);
\end{tikzpicture} & 62/2835 & 7936 & 496 \\
\hline
\begin{tikzpicture}[scale=0.6]
\draw [black, thick] (1,0)--(2,1)--(1,2)--(2,1)--(1,1);
\draw [black, thick] (2,1)--(3,1)--(4,2)--(3,3)--(4,2)--(5,3);
\draw [black, thick] (3,1)--(4,0)--(3,-1)--(4,0)--(5,-1);
\end{tikzpicture} & 113/2835 & 144640 & 9040 \\
\hline
\begin{tikzpicture}[scale=0.6]
\draw [black, thick] (1,0)--(2,1)--(1,2);
\draw [black, thick] (3,1)--(4,1);
\draw [black, thick] (2,1)--(3,1)--(4,2)--(3,3)--(4,2)--(5,3);
\draw [black, thick] (3,1)--(4,0)--(3,-1)--(4,0)--(5,-1);
\end{tikzpicture} & 452/14175 & 115712 & 7232 \\[0.5ex]
 \hline
\end{tabular}
\caption{Trees with at most 10 edges that are not path-trees.}
\label{table:2}
\end{table}

For the examples provided in the two tables, we used a code in Python, available at the address: \href{https://github.com/vodka4247/phylogenetic-degrees}{https://github.com/vodka4247/phylogenetic-degrees}, where we implemented the recurrence formula given in Theorem~\ref{v-splitting}.

\vspace{.2in}
\section*{Acknowledgement}
We thank Alexandra Vod\u a for helping us with Phython code implementations and Dumitru Stamate for carefully reading the paper.
RD was supported by the Alexander von Humboldt Foundation, and a grant of the Ministry of Research, Innovation and Digitization, CNCS - UEFISCDI, project number PN-III-P1-1.1-TE-2021-1633, within PNCDI III.
MV was supported by Slovak VEGA grant 1/0152/22.


\vspace{.2in}
\begin{thebibliography}{}
\vspace{.1in}

\bibitem{AR1} Elizabeth S. ~Allman, John A.~Rhodes, \textit{Phylogenetic invariants for the general Markov model of sequence mutation}, Mathematical biosciences, {\bf 186}(2) (2003), 113--144.\\

\bibitem{AR2} Elizabeth S. ~Allman, John A.~Rhodes, \textit{Quartets and parameter recovery for the general Markov model of sequence mutation}, Applied Mathematics Research Express, {\bf 2004}(4) (2004), 107--131.\\

 \bibitem{bw} Weronika~Buczy\'nska, Jaros{\l}aw~Wi\'sniewski, \textit{On geometry of binary symmetric models of phylogenetic trees}, J. Eur. Math. Soc., {\bf 9} (3) (2007), 609--635.\\
 
\bibitem{RM} Rodica ~Dinu, Martin~Vodička, \textit{Gorenstein property for phylogenetic trivalent trees}, J. Algebra, {\bf 575} (2021), 233--255.\\

\bibitem{RM-claw} Rodica Andreea ~Dinu, Martin~Vodička, \textit{Phylogenetic degrees for claw trees}, J. Comb. Theory Ser. A. {\bf 206} (2024), 105886.\\

\bibitem{RM-tripods} Rodica Andreea ~Dinu, Martin~Vodička, \textit{Classification of normal phylogenetic varieties for tripods}, arxiv preprint arXiv:2309.05301 (2023).\\

 \bibitem{draisma1} Jan~Draisma, Jochen~Kuttler, \textit{On the ideals of equivariant tree models}, Math. Ann., {\bf 344} (3) (2009), 619--644.\\

\bibitem{triangbook} Jesús~De~Loera, Jörg~Rambau, Francisco~Santos, \textit{Triangulations: structures for algorithms and applications}, Vol. 25. Springer Science \& Business Media, 2010.\\
 
 \bibitem{erss} Nicholas~Eriksson, Kristian~Ranestad, Bernd~Sturmfels, Seth~Sullivant, \textit{Phylogenetic algebraic geometry}, Projective varieties with Unexpected Properties, Siena, Italy (2004), 237--256.\\

  \bibitem{evans} Steven N.~Evans, Terrence ~P.~Speed, \textit{Invariants of some probability models used in phylogenetic inference}, Ann. Stat., {\bf 21} (1) (1993), 355--377.\\
 
 \bibitem{fels} Joseph~Felsenstein, \textit{Inferring phylogenies}, Sinauer Associates, Inc., Sunderland, 2003.\\

 \bibitem{fulton} William Fulton, \textit{Introduction to toric varieties}, {\bf 131}, Princeton university press, 1993.\\

 \bibitem{HP} M.D.~Hendy, D.~Penny, A framework for the quantitative study of evolutionary trees. \emph{Systematic zoology}, {\textbf 38}(1989), no.~4, 297--309.\\

\bibitem{kimura} Motoo~Kimura, \textit{Estimation of evolutionary distances between homologous nucleotide sequences}, Proceedings of the National Academy of Sciences, {\bf 78} (1) (1981), 454--458.\\

\bibitem{manon} Christopher~Manon, Coordinate rings for the moduli stack of $SL_2(C)$ quasi-parabolic principal bundles on a curve and toric fiber products. \emph{J. Algebra}, {\textbf 365}(2012), 163--183.\\

 \bibitem{h-rep} Marie~Mauhar, Joseph~Rusinko, Zoe~Vernon, \textit{H-representation of the Kimura-3 polytope for the m-claw tree}, SIAM J. Discrete Math., {\bf 31} (2) (2017), 783--795.\\

 \bibitem{matJCTA} Mateusz~Micha{\l}ek, \textit{Constructive degree bounds for group-based models}, J. Comb. Theory, Ser. A, {\textbf 120} (2013), no.~7, 1672--1694.\\

\bibitem{mateusz} Mateusz~Micha{\l}ek, \textit{Geometry of phylogenetic group-based models}, J. Algebra, (2011), 339--356.\\

\bibitem{mateusztor}  Mateusz~Micha{\l}ek, \textit{Toric varieties in phylogenetics}, Diss. Math., {\bf 511} (2015), 3--86.\\


\bibitem{matem}  Mateusz~Micha{\l}ek, Emanuele~Ventura, \textit{Finite phylogenetic complexity and combinatorics of tables}, Algebra \& Number Theory, {\bf 11}(1) (2017), 235--252.\\


\bibitem{examples} Luis David Garcia Punte, Marina Garrote-López, and Elima Shehu, \textit{Computing algebraic degrees of phylogenetic varieties}, Algebraic Statistics, {\bf 215} (2023).\\

\bibitem{ss} Bernd~Sturmfels, Seth~Sullivant, \textit{Toric ideals of phylogenetic invariants}, J. Computat. Biol., {\bf 12} (2005), 204--228.\\

\bibitem{seth} Seth~Sullivant, \textit{Toric fiber product}, J. Algebra,  {\bf 316} (2) (2007), 560--577.\\

\bibitem{sethbook} Seth~Sullivant, \textit{Algebraic statistics}, {\bf 194}, American Mathematical Soc., 2018.\\

\bibitem{sz} L.A.~Sz\'{e}kely, M.A.~Steel, P.L.~Erd\"os, \textit{Fourier calculus on evolutionary trees}, Adv. Appl. Math.,  {\bf 14} (2) (1993), 200--210.\\

\bibitem{martinko} Martin~Vodi\v{c}ka, \textit{Normality of the Kimura 3-parameter model}, SIAM J. Discrete Math., {\bf 35}(3) (2021), 1536--1547.


\end{thebibliography}
\end{document}